\documentclass[reqno]{amsart}
\usepackage{amsmath, amsthm, amssymb, amstext}

\usepackage{hyperref,xcolor}
\hypersetup{
    pdfborder={0 0 0},
    colorlinks,
}
\usepackage{enumitem}
\allowdisplaybreaks

\newtheorem{theorem}{Theorem}
\newtheorem{remark}[theorem]{Remark}

\newtheorem{proposition}[theorem]{Proposition}
\newtheorem{corollary}[theorem]{Corollary}
\newtheorem{definition}[theorem]{Definition}
\newtheorem{example}[theorem]{Example}

\DeclareMathOperator*{\divergenz}{div}              %
\DeclareMathOperator*{\Ss}{S}

\newcommand{\N}{\mathbb{N}}

\newcommand{\R}{\mathbb{R}}

\newcommand*\diff{\mathrm{d}}

\newcommand{\Lp}[1]{L^{#1}(\Omega)}

\newcommand{\Lprand}[1]{L^{#1}(\partial\Omega)}
\newcommand{\Wp}[1]{W^{1,#1}(\Omega)}

\newcommand{\Wpzero}[1]{W^{1,#1}_0(\Omega)}

\newcommand{\Om}{\Omega}

\newcommand{\into}{\int_{\Omega}}

\newcommand{\weak}{\rightharpoonup}

\newcommand{\Linf}{L^{\infty}(\Omega)}
\newcommand{\close}{\overline{\Omega}}

\renewcommand{\l}{\left}
\renewcommand{\r}{\right}
\newcommand{\Hi}{\mathcal H}
\newcommand{\WH}{W^{1, \mathcal{H}}(\Omega)}

\numberwithin{theorem}{section}
\numberwithin{equation}{section}

\newcommand{\U}{\mathcal{U}}
\newcommand{\ov}{\overline}
\newcommand{\ds}{\displaystyle}
\newcommand{\la}{\langle}
\newcommand{\ra}{\rangle}
\newcommand{\rh}{\rightharpoonup}
\newcommand{\al}{\alpha}

\newcommand{\pa}{\partial}
\newcommand{\Ga}{\Gamma}

\newcommand{\B}{{\mathcal B}}
\newcommand{\un}{\underline}
\newcommand{\K}{{\mathcal K}}

\renewcommand{\S}{{\mathcal S}}
\newcommand{\F}{{\mathcal F}}
\newcommand{\T}{{\mathcal T}}
\newcommand{\V}{{\mathcal V}}
\newcommand{\W}{{\mathcal W}}
\newcommand{\emb}{\hookrightarrow}

\newcommand{\q}{{q(\cdot)}}
\newcommand{\ep}{\varepsilon}

\title[Multi-valued variational inequalities for double phase problems]{Multi-valued variational inequalities for variable exponent double phase problems: comparison and extremality results}

\author[S.\,Carl]{Siegfried Carl}
\address[S.\,Carl]{Institut f\"ur Mathematik, Martin-Luther-Universit\"at Halle-Witten\-berg, 06099 Halle, Germany}
\email{siegfried.carl@mathematik.uni-halle.de}

\author[V.K.\,Le]{Vy Khoi Le}
\address[V.K.\,Le]{Department of Mathematics and Statistics, Missouri University of Science and Technology,
    Rolla, MO 65409, USA}
\email{vy@mst.edu}

\author[P.\,Winkert]{Patrick Winkert}
\address[P.\,Winkert]{Technische Universit\"{a}t Berlin, Institut f\"{u}r Mathematik, Stra\ss e des 17.\,Juni 136, 10623 Berlin, Germany}
\email{winkert@math.tu-berlin.de}


\begin{document}

\begin{abstract}
    We prove existence and comparison results for multi-valued variational inequalities in a bounded domain $\Omega$ of the form
    \begin{align*}
        u\in K\,:\, 0 \in Au+\partial I_K(u)+\mathcal{F}(u)+\mathcal{F}_\Gamma(u)\quad\text{in }W^{1, \mathcal{H}}(\Omega)^*,
    \end{align*}
    where  $A\colon W^{1, \mathcal{H}}(\Omega) \to W^{1, \mathcal{H}}(\Omega)^*$  given by
    \begin{align*}
        Au:=-\divergenz\l(|\nabla u|^{p(x)-2} \nabla u+ \mu(x) |\nabla u|^{q(x)-2} \nabla u\r)
    \end{align*}
    for $u \in W^{1, \mathcal{H}}(\Omega)$, is the double phase operator with variable exponents and $W^{1, \mathcal{H}}(\Omega)$ is the associated  Musielak-Orlicz Sobolev space. First, an existence result is proved under some weak coercivity condition. Our main focus aims at the treatment of the problem under consideration when coercivity fails. To this end we establish the method of sub-supersolution for the multi-valued variational inequality in the space $W^{1, \mathcal{H}}(\Omega)$  based on appropriately defined sub- and supersolutions, which yields the existence of solutions within an ordered interval of sub-supersolution.  Moreover, the existence of extremal solutions will be shown provided the closed, convex subset $K$ of $W^{1, \mathcal{H}}(\Omega)$ satisfies a lattice condition. As an application of the sub-supersolution method we are able to show that a class of generalized variational-hemivariational inequalities with a leading double phase operator are included as a special case of the multi-valued variational inequality considered here. Based on a fixed point argument, we also study the case when the corresponding Nemytskij operators $\mathcal{F}, \mathcal{F}_\Gamma$ need not be continuous. At the end, we give an example of the construction of sub- and supersolutions related to the problem above.
\end{abstract}

\maketitle

\section{Introduction and Main Results}

In this paper we prove comparison and extremality results for a wide class of multi-valued variational inequalities driven by the double phase operator with variable exponents. This operator, denoted by $A\colon \WH \to \WH ^*$, is given in the form
\begin{equation}
    \label{103}
    Au:=-\divergenz\l(|\nabla u|^{p(x)-2} \nabla u+ \mu(x) |\nabla u|^{q(x)-2} \nabla u\r)
\end{equation}
for $ u \in\WH$, where $p,q \in C(\close)$ with $1<p(x)<N$, $p(x)<q(x)$ for all $x\in\close$, $0 \leq \mu(\cdot) \in \Lp{1}$ is the weight function and $\WH$ is the corresponding Musielak-Orlicz Sobolev space (see Section \ref{section_2} for its precise definition). Note that \eqref{103} reduces to the $p(x)$-Laplacian when $\mu \equiv 0$ and to the $(p(x),q(x))$-Laplacian when $\inf \mu>0$.

When $p$ and $q$ are constants, such setting is originally due to Zhikov \cite{Zhikov-1986} who introduced and studied the integral functional
\begin{align}
    \label{integral_minimizer}
    \omega \mapsto \int \big(|\nabla  \omega|^p+\mu(x)|\nabla  \omega|^q\big)\,\diff x
\end{align}
in order to describe models for strongly anisotropic materials. The functional \eqref{integral_minimizer} also demonstrated its importance in the study of duality theory and in the context of the Lavrentiev phenomenon, see Zhikov \cite{Zhikov-1995}. Note that \eqref{integral_minimizer} is related to the differential operator
\begin{align}
    \label{double_phase_operator}
    u\mapsto-\divergenz\big(|\nabla u|^{p-2}\nabla u+\mu(x) |\nabla u|^{q-2}\nabla u\big),
\end{align}
which is a special case of \eqref{103}. From the physical point of view, \eqref{integral_minimizer} describes the phenomenon that the energy density changes its ellipticity and growth properties according to the point in the domain. In the elasticity theory, for example, the modulating coefficient $\mu(\cdot)$ dictates the geometry of composites made of two different materials with distinct power hardening exponents $q$ and $p$, see Zhikov \cite{Zhikov-2011}. From the mathematical point of view, the behavior of \eqref{integral_minimizer} depends on the sets on which the weight function $\mu(\cdot)$ vanishes or not. Therefore, we have two phases $(\mu(x)=0$ or $\neq 0$) and so we call it double phase. Even though no global regularity theory for double phase problems exists yet, there are some remarkable results about local minimizers, see \cite{ Baroni-Colombo-Mingione-2015, Baroni-Colombo-Mingione-2018,Benslimane-Aberqi-Bennouna-2021, Colombo-Mingione-2015a, Colombo-Mingione-2015b,De-Filippis-Mingione-2021,Liu-Pucci-2023, Marcellini-1991, Marcellini-1989b, Ragusa-Tachikawa-2020}. We also refer to the recent overview article in \cite{Mingione-Radulescu-2021}.

Let us next formulate the problem under consideration. To this end, let $\Omega\subset \R^N$ ($N\geq 2$) be a bounded domain with Lipschitz boundary $\partial\Omega$ and let $\Gamma \subset \partial\Omega$ be a relatively open subset and denote $\Gamma_0=\partial\Omega \setminus \Gamma$ such that $\partial\Omega =\Gamma \cup \Gamma_0$. We consider the multi-valued elliptic variational inequality of the form
\begin{equation}
    \label{problem}
    u\in K\,:\, 0 \in Au+\partial I_K(u)+\mathcal{F}(u)+\mathcal{F}_\Gamma(u)\quad\text{in }\WH ^*,
\end{equation}
where $\WH ^*$ its dual space of $\WH$, $K$ is a closed convex subset of the closed subspace $V_{\Gamma_0}$ of $\WH$  defined by
\begin{align*}
    V_{\Gamma_0}=\l\{u\in \WH \,:\, u\mid_{\Gamma_0}=0\r\},
\end{align*}
$I_K$ is the indicator function related to $K$, and $\partial I_K$ denotes its subdifferential. The lower order multi-valued operators $\mathcal{F}$ and $\mathcal{F}_\Gamma$ are generated by the multi-valued functions $f\colon \Omega\times\R\to 2^{\R}\setminus\{\emptyset\}$ and $f_\Gamma\colon\Gamma \times\R\to 2^{\R}\setminus\{\emptyset\}$, respectively. Let
\begin{align}
    \label{critical_exponent}
    p^*(x):=\frac{Np(x)}{N-p(x)} \quad \text{and}\quad p_*(x):=\frac{(N-1)p(x)}{N-p(x)}
    \quad\text{for all }x\in\close
\end{align}
be the critical exponents to $p$ for $1<p(x)<N$, and denote by  $p'(\cdot)$ the H\"older conjugate to $p$ given by $p'(\cdot)=\frac{p(\cdot)}{p(\cdot)-1}$.  We assume the following hypotheses:

\begin{enumerate}
    \item[\textnormal{(H0)}]
        $p,q \in C(\close)$ such that $1<p(x)<N$, $p(x) < q(x)<p^*(x)$ and $0 \leq \mu(\cdot) \in \Lp{\infty}$.
    \item[\textnormal{(F1)}]
        $f\colon \Omega\times\R\to 2^{\R}\setminus\{\emptyset\}$ and $f_\Gamma\colon\Gamma \times\R\to 2^{\R}\setminus\{\emptyset\}$ are graph measurable on $\Omega\times\R$ and $\Gamma\times\R$, respectively, and for a.\,a.\,$x\in\Omega$ the function $f(x,\cdot)\colon \R\to2^{\R}$ is upper semicontinuous and for a.\,a.\,$x\in\Gamma$, the function $f_\Gamma(x,\cdot)\colon\R\to 2^{\R}$ is upper semicontinuous.
    \item[\textnormal{(F2)}]
        There exist $r_1\in C(\close)$, $r_2 \in C(\Gamma)$ with $1<r_1(x)<p^*(x)$ for all $x\in \close$, $1<r_2(x)<p_*(x)$ for all $x\in\Gamma$, $\beta\geq 0$, $\beta_\Gamma\geq 0$ and functions $\alpha\in\Lp{r_1'(\cdot)}$, $\alpha_\Gamma \in L^{r_2'(\cdot)}(\Gamma)$  such that
        \begin{align*}
            \sup \l\{|\eta| \,:\, \eta \in f(x,s)\r\} \leq \alpha(x)+\beta |s|^{r_1(x)-1}
        \end{align*}
        for a.\,a.\,$x\in\Omega$, for all $s\in\R$, and
        \begin{align*}
            \sup \l\{|\zeta| \,:\, \zeta \in f_\Gamma(x,s)\r\} \leq \alpha_\Gamma(x)+\beta_\Gamma |s|^{r_2(x)-1}
        \end{align*}
        for a.\,a.\,$x\in\Gamma$, and for all $s\in\R$.
\end{enumerate}

We point out that the classical obstacle problem fits in our setting, that is,
\begin{align*}
    K=\l\{u\in\WH \,:\, u(x)\geq \psi(x)\text{ a.\,e.\,in }\Omega\r\}
\end{align*}
with a given obstacle $\psi\colon \Omega\to\R$. Originally, the study of obstacle problems is due the pioneering contribution by Stefan \cite{Stefan-1889} in which the temperature distribution in a homogeneous medium undergoing a phase change, typically a body of ice at zero degrees centigrade submerged in water, was studied. Furthermore, we mention the famous work of J.-L.~Lions \cite{Lions-1969} who studied the equilibrium position of an elastic membrane which lies above a given obstacle and which turns out as the unique solution of the Dirichlet energy functional minimized on the closed convex set $K$.

Before we state our main results, we first give the definition of a weak solution to problem \eqref{problem}.
\begin{definition}
    \label{D101}
    A function $u\in K$ is said to be a (weak) solution of \eqref{problem} if there exist $\tau_1\in C(\close)$, $\tau_2\in C(\Gamma)$, $1<\tau_1(x)<p^*(x)$ for all $x\in\close$, $1<\tau_2(x)<p_*(x)$ for all $x\in\Gamma$ and $\eta\in\Lp{\tau'_1(\cdot)}$, $\zeta\in L^{\tau'_2(\cdot)}(\Gamma)$ such that $\eta(x) \in f(x,u(x))$ for a.\,a.\,$x\in\Omega$, $\zeta(x) \in f_\Gamma(x,u(x))$ for a.\,a.\,$x\in\Gamma$ and
    \begin{align}\label{problem1}
        \begin{split}
            & \into \l(|\nabla u|^{p(x)-2} \nabla u+ \mu(x) |\nabla u|^{q(x)-2} \nabla u\r) \cdot \nabla (v-u)\,\diff x\\
            &+\into \eta(v-u) \,\diff x+\int_\Gamma \zeta(v-u)\,\diff \sigma \geq 0
        \end{split}
    \end{align}
    for all $v\in K$.
\end{definition}

Note, for simplicity of notation, the boundary integral $\int_\Gamma \zeta(v-u)\,\diff \sigma$ stands for
\begin{align*}
    \int_\Gamma \zeta\big(i_{\tau_2(\cdot)}v|_{\Gamma}-i_{\tau_2(\cdot)}u|_{\Gamma}\big)\,\diff \sigma,
\end{align*}
where $i_{\tau_2(\cdot)}\colon \WH\to L^{\tau_2(\cdot)}(\partial\Omega)$ denotes the trace operator, and $i_{\tau_2(\cdot)}v|_{\Gamma}$ is the restriction of $i_{\tau_2(\cdot)}v$ to $\Gamma$.

The multi-valued variational inequality \eqref{problem} covers a wide range of elliptic problems which can be deduced from \eqref{problem} by specifying $\Gamma,$ $K$, and the lower order terms. To give an idea, let us consider a few examples.

\begin{example}
    \label{example_1}
    If $\Gamma=\partial\Omega$, then $\Gamma_0=\emptyset$,  and $V_{\Gamma_0}=\WH$. If $K=\WH$, then \eqref{problem} reduces to the following multi-valued elliptic boundary value problem
    \begin{equation*}
        \begin{aligned}
            -\divergenz\l(|\nabla u|^{p(x)-2} \nabla u+ \mu(x) |\nabla u|^{q(x)-2} \nabla u\r) +f(x,u) & \ni 0             &  & \text{in } \Omega,          \\
            -\frac{\partial u}{\partial \nu_A}                                                         & \in f_\Gamma(x,u) &  & \text{on } \partial\Omega,
        \end{aligned}
    \end{equation*}
    where
    \begin{align*}
        \frac{\partial u}{\partial \nu_A}=\l(|\nabla u|^{p(x)-2} \nabla u+ \mu(x) |\nabla u|^{q(x)-2} \nabla u\r)\cdot \nu
    \end{align*}
    with $\nu$ denoting the outward unit normal at $\Gamma$.
\end{example}

\begin{example}
    If $\Gamma_0=\partial\Omega$, then $\Gamma=\emptyset$, and $V_{\Gamma_0}=W^{1,\mathcal{H}}_0(\Omega)$. If $K=W^{1,\mathcal{H}}_0(\Omega)$, then \eqref{problem} becomes the following multi-valued Dirichlet boundary value problem
    \begin{equation*}
        \begin{aligned}
            -\divergenz\l(|\nabla u|^{p(x)-2} \nabla u+ \mu(x) |\nabla u|^{q(x)-2} \nabla u\r) +f(x,u) & \ni 0 &  & \text{in } \Omega,          \\
            u                                                                                          & =0    &  & \text{on } \partial\Omega.
        \end{aligned}
    \end{equation*}
\end{example}

Further special cases can be deduced from \eqref{problem} such as a mixed boundary value problems that arise when $|\Gamma|>0$ and $|\Gamma_0|>0$,  and $K=V_{\Gamma_0}$. In Section \ref{section_6} we will see that \eqref{problem} also includes an important class of   generalized variational-hemivariational inequalities.

Our first result is the following existence theorem for \eqref{problem} under a coercivity condition.

\begin{theorem}
    \label{T101}
    Let hypotheses \textnormal{(H0)}, \textnormal{(F1)} and \textnormal{(F2)} be satisfied and suppose  the following coercivity condition holds:
    \begin{enumerate}
        \item[]
            There exist $u_0\in K$ and $R\ge \|u_0\|_{1,\mathcal{H}}$ such that $K \cap B_R(0) \not=\emptyset$ and
            \begin{equation}
                \label{1-01}
                \la A u  + \eta^* +\zeta^*   , u-u_0\ra > 0 ,
            \end{equation}
            for all $u\in  K$ with $\|u \|_{1,\mathcal{H}}=R$, for all $\eta^* \in \F(u)$ and for all $\zeta^* \in \F_\Gamma(u)$.
    \end{enumerate}
    Then problem \eqref{problem} has at least one solution in the sense of Definition \ref{D101}.
\end{theorem}

The proof will be given in Section \ref{section_3}, see also Corollary \ref{C101}, which is a direct consequence of Theorem \ref{T101}. If the coercivity condition \eqref{1-01} or appropriate generalized versions of coercivity are not satisfied then problem \eqref{problem} may have no solutions. However, in the noncoercive case we still are able to prove the existence of solutions provided appropriately defined sub-supersolutions for \eqref{problem} exist. In this paper we establish an sub-supersolution method based on the following definition of sub- and supersolutions of problem \eqref{problem}. For functions $u,v \colon \Omega \to \R$ we use the notation $u \wedge v=\min(u,v), u \vee v=\max(u,v), K \wedge K = \{u \wedge v\,:\, u,v \in K \}, K \vee K = \{u \vee v\,:\, u,v \in K \}$ and $u \wedge K = \{u\} \wedge K, u \vee K = \{u\}\vee K$.

\begin{definition}
    \label{D102}
    A function $\underline{u}\in \Wp{\mathcal{H}}$ is said to be a (weak) subsolution of \eqref{problem} if there exist $\tau_1\in C(\close)$, $\tau_2\in C(\Gamma)$,  $1<\tau_1(x)<p^*(x)$ for all $x\in\close$, $1<\tau_2(x)<p_*(x)$ for all $x\in\Gamma$ and $\underline{\eta}\in\Lp{\tau'_1(\cdot)}$, $\underline{\zeta}\in L^{\tau'_2(\cdot)}(\Gamma)$ such that
    \begin{enumerate}
        \item[\textnormal{(i)}]
            $\underline{u} \vee K\subset K$;
        \item[\textnormal{(ii)}]
            $\underline{\eta}(x) \in f(x,\underline{u}(x))$ for a.\,a.\,$x\in\Omega$, $\underline{\zeta}(x) \in f_\Gamma(x,\underline{u}(x))$ for a.\,a.\,$x\in\Gamma$;
        \item[\textnormal{(iii)}]
            \begin{align*}
                 & \into \l(|\nabla \underline{u}|^{p(x)-2} \nabla \underline{u}+ \mu(x) |\nabla \underline{u}|^{q(x)-2} \nabla \underline{u}\r) \cdot \nabla (v-\underline{u})\,\diff x\\
                & +\into \underline{\eta}(v-\underline{u}) \,\diff x+\int_\Gamma \underline{\zeta}(v-\underline{u})\,\diff \sigma \geq 0
            \end{align*}
            for all $v\in \underline{u} \wedge K$.
    \end{enumerate}
\end{definition}

\begin{definition}
    \label{D103}
    A function $\overline{u}\in \Wp{\mathcal{H}}$ is said to be a (weak) supersolution of \eqref{problem} if there exist $\tau_1\in C(\close)$, $\tau_2\in C(\Gamma)$,  $1<\tau_1(x)<p^*(x)$ for all $x\in\close$, $1<\tau_2(x)<p_*(x)$ for all $x\in\Gamma$ and $\overline{\eta}\in\Lp{\tau'_1(\cdot)}$, $\overline{\zeta}\in L^{\tau'_2(\cdot)}(\Gamma)$ such that
    \begin{enumerate}
        \item[\textnormal{(i)}]
            $\overline{u} \wedge K\subset K$;
        \item[\textnormal{(ii)}]
            $\overline{\eta}(x) \in f(x,\overline{u}(x))$ for a.\,a.\,$x\in\Omega$, $\overline{\zeta}(x) \in f_\Gamma(x,\overline{u}(x))$ for a.\,a.\,$x\in\Gamma$;
        \item[\textnormal{(iii)}]
            \begin{align*}
                 & \into \l(|\nabla \overline{u}|^{p(x)-2} \nabla \overline{u}+ \mu(x) |\nabla \overline{u}|^{q(x)-2} \nabla \overline{u}\r) \cdot \nabla (v-\overline{u})\,\diff x\\
                &+\into \overline{\eta}(v-\overline{u}) \,\diff x+\int_\Gamma \overline{\zeta}(v-\overline{u})\,\diff \sigma \geq 0
            \end{align*}
            for all $v\in \overline{u} \vee K$.
    \end{enumerate}
\end{definition}

\begin{remark}
    We note that although variational inequalities are generally nonsymmetric due to the presence of constraints, the notions for sub- and supersolution defined by Definition \ref{D102} and Definition \ref{D103}, respectively, do have a symmetric structure in the following sense:  one obtains the definition for the supersolution $\overline{u}$ from the definition of the subsolution by replacing $\underline{u},\underline{\eta}, \underline{\zeta} $ in the definition of subsolution by $\overline{u}, \overline{\eta}, \overline{\zeta}$, and interchanging $\vee$ by $\wedge$.  Symmetric structure is a main feature of the sub-supersolution concepts for smooth equations, which has been extended here to multi-valued variational inequalities with variable exponent double-phase operator.

    Just for illustration, let us apply the above definitions to the special case given by Example \ref{example_1} and assume that $f$ and $f_\Gamma$ are single-valued, that is
    \begin{equation}
        \label{example1}
        \begin{aligned}
            -\divergenz\l(|\nabla u|^{p(x)-2} \nabla u+ \mu(x) |\nabla u|^{q(x)-2} \nabla u\r) +f(x,u) & = 0 &  & \text{in } \Omega,          \\
            \frac{\partial u}{\partial \nu_A}+ f_\Gamma(x,u)                                           & = 0 &  & \text{on } \partial\Omega.
        \end{aligned}
    \end{equation}
    Let $\underline{u}$ be a subsolution according to Definition \ref{D102}. As $K=\WH$ and $\WH$ has lattice structure, condition \textnormal{(i)} is trivially satisfied. Condition \textnormal{(ii)} yields $\underline{\eta}(x)= f(x,\underline{u}(x))$ for a.\,a.\,$x\in\Omega$ and $\underline{\zeta}(x)= f_\Gamma(x,\underline{u}(x))$ for a.\,a.\,$x\in\Gamma$. For any $\varphi\in K=\WH$ we test \textnormal{(iii)} with $v=\underline{u}\wedge \varphi=\underline{u}-(\underline{u}-\varphi)^+$ which results in
    \begin{align*}
         & \into \l(|\nabla \underline{u}|^{p(x)-2} \nabla \underline{u}+ \mu(x) |\nabla \underline{u}|^{q(x)-2} \nabla \underline{u}\r) \cdot \nabla (\underline{u}-\varphi)^+\,\diff x \\
         & +\into f(x,\underline{u}(x))(\underline{u}-\varphi)^+ \,\diff x
        +\int_\Gamma f_\Gamma(x,\underline{u}(x))(\underline{u}-\varphi)^+\,\diff \sigma \le 0
    \end{align*}
    for all $\varphi\in \WH$. Since the set $\{(\underline{u}-\varphi)^+\,:\, \varphi\in \WH\}$ equals $\{\psi\in \WH\,:\, \psi\ge 0\}$, the last inequality is nothing but the usual notion of subsolution for the boundary value problem (\ref{example1}), that is,
    \begin{equation*}
        \begin{aligned}
            -\divergenz\l(|\nabla \underline{u}|^{p(x)-2} \nabla \underline{u}+ \mu(x)\underline{u}(x) |\nabla \underline{u}|^{q(x)-2} \nabla \underline{u}\r) +f(x,\underline{u}) & \leq 0 &  & \text{in } \Omega,          \\
            \frac{\partial \underline{u}}{\partial \nu_A}+ f_\Gamma(x,\underline{u})                                                                                               & \leq 0 &  & \text{on } \partial\Omega.
        \end{aligned}
    \end{equation*}
    Similarly, Definition \ref{D103} for the supersolution $\overline{u}$ of \eqref{example1} reduces to
    \begin{equation*}
        \begin{aligned}
            -\divergenz\l(|\nabla \overline{u}|^{p(x)-2} \nabla \overline{u}+ \mu(x)\overline{u}(x) |\nabla \overline{u}|^{q(x)-2} \nabla \overline{u}\r) +f(x,\overline{u}) & \geq 0 &  & \text{in } \Omega,          \\
            \frac{\partial \overline{u}}{\partial \nu_A}+ f_\Gamma(x,\overline{u})                                                                                           & \geq 0 &  & \text{on } \partial\Omega.
        \end{aligned}
    \end{equation*}
\end{remark}

Next, we suppose the following local boundedness conditions on the multi-valued nonlinearities with respect to the order interval $[\underline{u} ,\overline{u} ]$.
\begin{enumerate}
    \item[\textnormal{(F3)}]
        Let $\underline{u}$ and $\overline{u}$ be sub- and supersolutions of \eqref{problem} such that $\underline{u}\leq \overline{u}$ and suppose the following growth conditions
        \begin{align*}
            \sup \l\{|\eta| \,:\, \eta \in f(x,s)\r\} \leq k_\Omega(x)\quad\text{for a.\,a.\,}x\in\Omega, \\
            \sup \l\{|\zeta| \,:\, \zeta \in f_\Gamma(x,s)\r\} \leq k_\Gamma(x)\quad\text{for a.\,a.\,}x\in\Gamma,
        \end{align*}
        for all $s\in [\underline{u}(x),\overline{u}(x)]$ and for some $k_\Omega\in L^{\tau_1'(\cdot)}(\Omega)$, $k_\Gamma \in L^{\tau_2'(\cdot)}(\Gamma)$.
\end{enumerate}
The sub-supersolution method for \eqref{problem} is established by the following existence and comparison result.
\begin{theorem}
    \label{T102}
    Let $\underline{u}$ and $\overline{u}$ be an ordered pair of sub- and supersolutions of \eqref{problem} fulfilling $\underline{u} \leq \overline{u}$ and let hypotheses \textnormal{(H0)}, \textnormal{(F1)} and \textnormal{(F3)} be satisfied. Then problem \eqref{problem} has a solution $u\in K$ such that $\underline{u} \leq u \le\overline{u}$ a.e.\,in $\Omega$.
\end{theorem}

We remark that Theorem \ref{T102} will be seen as straightforward consequence of a general existence and comparison principle (see Theorem \ref{T4-1}) which will be proved in Section \ref{section_4}, and which at the same time allows us to order-theoretically and topologically characterize the solution set $\mathcal{S}$  of all solutions of \eqref{problem} lying within the interval $[\underline{u},\overline{u}]$. We have the following characterization of $\mathcal{S}$, see Section \ref{section_5}.
\begin{theorem}
    \label{T103} $~$
    \begin{enumerate}
        \item[\textnormal{(i)}]
            Under the assumptions of Theorem \ref{T102}, the solutions set $\mathcal{S}$ is a compact subset of $W^{1,\Hi}(\Om)$.
        \item[\textnormal{(ii)}]
            If
            \begin{equation}
                \label{lattice}
                \S \wedge K \subset K \quad\text{and}\quad \S \vee K \subset K,
            \end{equation}
            then
            \begin{enumerate}
                \item[\textnormal{(a)}]
                    any $u\in \S$ is both a (weak) subsolution and supersolution of \eqref{problem}, and
                \item[\textnormal{(b)}]
                    $\S$ is directed both downward and upward, that is, for all $u_1, u_2\in \S$, there exists $w_1, w_2\in \S$ such that
                    \begin{align*}
                        w_1\leq\min\{ u_1, u_2\} \quad\text{and}\quad w_2\geq\max\{ u_1, u_2\}.
                    \end{align*}
            \end{enumerate}
        \item[\textnormal{(iii)}]
            If (\ref{lattice}) hold then $\S$ has smallest and greatest elements, that is, there are $u_*, u^*\in \S$ such that $u_* \le u \le u^*$ for all $u\in \S$.
    \end{enumerate}
\end{theorem}

In Section \ref{section_6}, as an application of  Theorem \ref{T102} and Theorem \ref{T103}, we are going to show that a class of generalized variational-hemivariational inequalities with the double phase operator as the leading operator of the form
\begin{equation}\label{106}
    \begin{aligned}
        u\in K\,:\, &\langle  Au, v-u\rangle+\int_{\Omega} j^\circ(\cdot,u,u; v-u)\,\diff x\\
        &+\int_{\Gamma} j_{\Gamma}^\circ(\cdot,u, u;  v- u)\,\diff \sigma\ge 0 \quad\text{for all } v\in K,
    \end{aligned}
\end{equation}
turn out to be a special case of \eqref{problem} only, see Theorem \ref{T601}. In Section \ref{section_7} we also study the case when the functions $f$ and $f_\Gamma$ need not be continuous (so $\F$ and $\F_\Gamma$ need not be pseudomonotone anymore). The idea in the proof is the usage of an fixed point argument, see Theorems \ref{theorem-fixed-point} and  \ref{theorem-discontinuous}.
Lastly, in Section \ref{section_8}, we construct nontrivial sub- and supersolutions of \eqref{problem} which can be applied to our results, see Theorem \ref{T-O} and Corollary \ref{C-O1}.

To the best of our knowledge, our results are new even in the case when $p$ and $q$ are constants. For double phase problems with variable exponents there are only few works, we mention the papers of \cite{Bahrouni-Radulescu-Winkert-2020} for the variable exponent  Baouendi-Grushin operator, of \cite{CGHW} for single-valued convection problems and  of \cite{Zeng-Radulescu-Winkert-2021} for obstacle problems involving multi-valued reaction terms with gradient dependence. Papers dealing with the constant exponent double phase \eqref{double_phase_operator} along with multi-valued right-hand sides can be found in \cite{Zeng-Bai-Gasinski-Winkert-2020} and \cite{Zeng-Gasinski-Winkert-Bai-2021} who studied obstacle problems involving the special case of Clarke's generalized gradients. Note that all these works are dealing with the coercive case.

Finally, we mention some recent results for single-valued double phase problems without constraints, such as,
\cite{Colasuonno-Squassina-2016} for eigenvalue problems for double phase problems, \cite{Gasinski-Papageorgiou-2019} for sign-changing solutions based on the Nehari manifold, \cite{Gasinski-Winkert-2020b} for general convection problems, \cite{Liu-Dai-2018} for superlinear double phase problems, \cite{Perera-Squassina-2018} for double phase problems via Morse theory, \cite{Shi-Radulescu-Repovs-Zhang-2020} for multiple solutions for double phase variational problems and \cite{Zhang-Radulescu-2018} for anisotropic double phase problems. As for multi-valued variational inequalities with leading $p$-Laplacian type operators we refer to \cite{Carl-Le-2015} for bounded domains, and \cite{Carl-Le-2019, Carl-Le-2022} for unbounded domains.

\section{Preliminaries}\label{section_2}

In this section we recall some results about variable exponent Sobolev space, Musielak-Orlicz Sobolev spaces and properties of the variable exponent double phase operator. The results are mainly taken from the books of \cite{Diening-Harjulehto-Hasto-Ruzicka-2011} and \cite{Harjulehto-Hasto-2019} as well as the papers of \cite{CGHW}, \cite{Fan-Zhao-2001} and  \cite{Kovacik-Rakosnik-1991}.

Let $\Omega\subset \R^N$ be a bounded domain with Lipschitz boundary $\partial \Omega$ and denote by $M(\Omega)$ the space of all measurable functions $u\colon\Omega\to\R$. Let $C_+(\overline\Omega)$ be a subset of $C(\overline\Omega)$ defined by
\begin{align*}
    C_+(\overline \Omega):=\l\{h\in C(\overline \Omega)\,:\,1<h(x)\text{ for all }x\in\overline \Omega\r\}.
\end{align*}
For any $r\in C_+(\overline \Omega)$, we define
\begin{align*}
    r_-:=\min_{x\in\overline \Omega}r(x)
    \quad \text{and}\quad
    r_+:=\max_{x\in\overline \Omega}r(x)
\end{align*}
and $r'\in C_+(\overline \Omega)$ stands for the conjugate variable exponent to $r$, namely,
\begin{align*}
    \frac{1}{r(x)}+\frac{1}{r'(x)}=1
    \quad\text{for all }x\in \overline \Omega.
\end{align*}

For $r\in C_+(\overline \Omega)$ fixed, the variable exponent Lebesgue space $L^{r(\cdot)}(\Omega)$ is defined by
\begin{align*}
    L^{r(\cdot)}(\Omega)=\left\{u\in M(\Omega)\,:\,\int_\Omega|u|^{r(x)}\,\diff x<+\infty\right\},
\end{align*}
equipped with the Luxemburg norm
\begin{align*}
    \|u\|_{r(\cdot)}:=\inf\left\{\lambda>0\,:\,\int_\Omega\left(\frac{|u|}{\lambda}\right)^{r(x)}\,\diff x\le 1\right\}.
\end{align*}
It is well-known that $L^{r(\cdot)}(\Omega)$ is a separable and reflexive Banach space. Furthermore, the dual space of $L^{r(\cdot)}(\Omega)$ is $L^{r'(\cdot)}(\Omega)$ and the following H\"older type inequality holds
\begin{align*}
    \int_\Omega|uv|\,\mathrm{d}x\le \left[\frac{1}{r_-}+\frac{1}{r'_-}\right]\|u\|_{r(\cdot)}\|v\|_{r'(\cdot)}\le 2\|u\|_{r(\cdot)}\|v\|_{r'(\cdot)}
\end{align*}
for all $u\in L^{r(\cdot)}(\Omega)$ and for all $v\in L^{r'(\cdot)}(\Omega)$. For $r_1,r_2\in C_+(\overline\Omega)$ with $r_1(x)\le r_2(x)$ for all $x\in\overline \Omega$ we have the continuous embedding
\begin{align*}
    L^{r_2(\cdot)}(\Omega)\hookrightarrow L^{r_1(\cdot)}(\Omega).
\end{align*}

In the same way, for any $\Gamma \subset \partial\Omega$, we define boundary variable exponent Sobolev spaces $L^{r(\cdot)}(\Gamma)$ with $r\in C(\Gamma)$, $r(x)>1$ for all $x\in\Gamma$ and norm $\|\cdot\|_{r(\cdot),\Gamma}$.

For any $r\in C_+(\overline \Omega)$, we consider the modular function $\rho_{r(\cdot)}\colon L^{r(\cdot)}(\Omega)\to \R$ given by
\begin{align}
    \label{defmodularf}
    \rho_{r(\cdot)}(u)=\int_\Omega|u|^{r(x)}\,\diff x\quad\text{for all $u\in L^{r(\cdot)}(\Omega)$}.
\end{align}
The following proposition states some important relations between the norm of $L^{r(\cdot)}(\Omega)$ and the modular function $\rho_{r(\cdot)}$ defined in \eqref{defmodularf}.
\begin{proposition}
    If $r\in C_+(\close)$ and $u, u_n\in \Lp{r(\cdot)}$, then we have the following assertions:
    \begin{enumerate}
        \item[\textnormal{(i)}]
            $\|u\|_{r(\cdot)}=\lambda \quad\Longleftrightarrow\quad \rho_{r(\cdot)}\l(\frac{u}{\lambda}\r)=1$ with $u \neq 0$;
        \item[\textnormal{(ii)}]
            $\|u\|_{r(\cdot)}<1$ (resp. $=1$, $>1$) $\quad\Longleftrightarrow\quad \rho_{r(\cdot)}(u)<1$ (resp. $=1$, $>1$);
        \item[\textnormal{(iii)}]
            $\|u\|_{r(\cdot)}<1$ $\quad\Longrightarrow\quad$ $\|u\|_{r(\cdot)}^{r_+} \leq \rho_{r(\cdot)}(u) \leq \|u\|_{r(\cdot)}^{r_-}$;
        \item[\textnormal{(iv)}]
            $\|u\|_{r(\cdot)}>1$ $\quad\Longrightarrow\quad$ $\|u\|_{r(\cdot)}^{r_-} \leq \rho_{r(\cdot)}(u) \leq \|u\|_{r(\cdot)}^{r_+}$;
        \item[\textnormal{(v)}]
            $\|u_n\|_{r(\cdot)} \to 0 \quad\Longleftrightarrow\quad\rho_{r(\cdot)}(u_n)\to 0$;
        \item[\textnormal{(vi)}]
            $\|u_n\|_{r(\cdot)}\to +\infty \quad\Longleftrightarrow\quad \rho_{r(\cdot)}(u_n)\to +\infty$.
    \end{enumerate}
\end{proposition}

For  $r\in C_+(\overline \Omega)$, we denote by $W^{1,r(\cdot)}(\Omega)$ the variable exponent Sobolev space given by
\begin{align*}
    W^{1,r(\cdot)}(\Omega)=\left\{u\in L^{r(\cdot)}(\Omega)\,:\,|\nabla u|\in L^{r(\cdot)}(\Omega)\right\}.
\end{align*}
We know that $W^{1,r(\cdot)}(\Omega)$ equipped with the norm
\begin{align*}
    \|u\|_{1,r(\cdot)}=\|u\|_{r(\cdot)}+\|\nabla u\|_{r(\cdot)}\quad\text{for all }u\in W^{1,r(\cdot)}(\Omega)
\end{align*}
is a separable and reflexive Banach space, where $\|\nabla u\|_{r(\cdot)}:=\|\,|\nabla u|\,\|_{r(\cdot)}$. We also consider the subspace $W_0^{1,r(\cdot)}(\Omega)$ of $W^{1,r(\cdot)}(\Omega)$ defined  by
\begin{align*}
    W_0^{1,r(\cdot)}(\Omega)=\overline{ C_0^\infty(\Omega)}^{\|\cdot\|_{1,r(\cdot)}}.
\end{align*}
From Poincar\'e's inequality we know that we can endow the space $W_0^{1,r(\cdot)}(\Omega)$ with the equivalent norm
\begin{align*}
    \|u\|_{1,r(\cdot),0}=\|\nabla u\|_{r(\cdot)}\quad\text{for all }u\in W_0^{1,r(\cdot)}(\Omega).
\end{align*}

We suppose now condition \textnormal{(H0)} and introduce the nonlinear function $\mathcal{H} \colon \Omega \times [0,\infty) \to [0,\infty)$ defined by
\begin{align*}
    \mathcal{H}(x,t):=t^{p(x)} +\mu(x)t^{q(x)} \quad\text{for all } (x,t)\in \Omega \times [0,\infty).
\end{align*}
Then, the corresponding Musielak-Orlicz space $\Lp{\mathcal{H}}$ is given
by
\begin{align*}
    L^{\mathcal{H}}(\Omega)=\left \{u \in M(\Omega) \,:\,\rho_{\mathcal{H}}(u) < +\infty \right \},
\end{align*}
endowed with the norm
\begin{align*}
    \|u\|_{\mathcal{H}} = \inf \left \{ \tau >0 : \rho_{\mathcal{H}}\left(\frac{u}{\tau}\right) \leq 1  \right \},
\end{align*}
where the related modular to $\mathcal{H}$ is given by
\begin{align*}
    \rho_{\mathcal{H}}(u) = \into \mathcal{H} (x,|u|)\,\diff x.
\end{align*}
The corresponding Musielak-Orlicz Sobolev space $W^{1,\mathcal{H}}(\Omega)$ is defined by
\begin{align*}
    W^{1,\mathcal{H}}(\Omega)= \Big \{u \in L^\mathcal{H}(\Omega) \,:\, |\nabla u| \in L^{\mathcal{H}}(\Omega) \Big\}
\end{align*}
equipped with the norm
\begin{align*}
    \|u\|_{1,\mathcal{H}}= \|\nabla u \|_{\mathcal{H}}+\|u\|_{\mathcal{H}},
\end{align*}
where $\|\nabla u\|_\mathcal{H}=\|\,|\nabla u|\,\|_{\mathcal{H}}$. Moreover, we denote by $W^{1,\mathcal{H}}_0(\Omega)$ the completion of $C^\infty_0(\Omega)$ in $W^{1,\mathcal{H}}(\Omega)$. We equip the space $\Wpzero{\mathcal{H}}$ with the equivalent norm
\begin{align*}
    \|u\|_{1,\mathcal{H},0}=\|\nabla u\|_{\mathcal{H}} \quad\text{for all }u \in \Wpzero{\mathcal{H}},
\end{align*}
see \cite[Proposition 2.18]{CGHW}.
We know that the spaces $L^{\mathcal{H}}(\Omega)$, $W^{1,\mathcal{H}}_0(\Omega)$  and $W^{1,\mathcal{H}}(\Omega)$ are reflexive Banach spaces, see \cite[Proposition 2.12]{CGHW}.

The next proposition shows the relation between the norm $\|\cdot\|_\mathcal{H}$ and the modular $\rho_{\mathcal{H}}$, see \cite[Proposition 2.13]{CGHW}.

\begin{proposition}
    Let hypotheses \textnormal{(H0)} be satisfied. Then the following holds:
    \begin{enumerate}
        \item[\textnormal{(i)}]
            If $u\neq 0$, then $\|u\|_{\mathcal{H}}=\lambda$ if and only if $ \rho_{\mathcal{H}}(\frac{u}{\lambda})=1$;
        \item[\textnormal{(ii)}]
            $\|u\|_{\mathcal{H}}<1$ (resp.\,$>1$, $=1$) if and only if $ \rho_{\mathcal{H}}(u)<1$ (resp.\,$>1$, $=1$);
        \item[\textnormal{(iii)}]
            If $\|u\|_{\mathcal{H}}<1$, then $\|u\|_{\mathcal{H}}^{q_+}\leqslant \rho_{\mathcal{H}}(u)\leqslant\|u\|_{\mathcal{H}}^{p_-}$;
        \item[\textnormal{(iv)}]
            If $\|u\|_{\mathcal{H}}>1$, then $\|u\|_{\mathcal{H}}^{p_-}\leqslant \rho_{\mathcal{H}}(u)\leqslant\|u\|_{\mathcal{H}}^{q_+}$;
        \item[\textnormal{(v)}]
            $\|u\|_{\mathcal{H}}\to 0$ if and only if $ \rho_{\mathcal{H}}(u)\to 0$;
        \item[\textnormal{(vi)}]
            $\|u\|_{\mathcal{H}}\to +\infty$ if and only if $ \rho_{\mathcal{H}}(u)\to +\infty$.
        \item[\textnormal{(vii)}]
            $\|u\|_{\mathcal{H}}\to 1$ if and only if $ \rho_{\mathcal{H}}(u)\to 1$.
        \item[\textnormal{(viii)}]
            If $u_n \to u$ in $\Lp{\mathcal{H}}$, then $\rho_{\mathcal{H}} (u_n) \to \rho_{\mathcal{H}} (u)$.
    \end{enumerate}
\end{proposition}

We now equip the space $\Wp{\mathcal{H}}$ with the equivalent norm
\begin{align*}
    \|u\|_{\hat{\rho}_\mathcal{H}}:=\inf & \Biggl\{\lambda >0 \,:\,
    \int_\Omega
    \Biggl[\l|\frac{\nabla u}{\lambda}\r|^{p(x)}+\mu(x)\l|\frac{\nabla u}{\lambda}\r|^{q(x)} \\
    &\qquad\qquad\qquad +\l|\frac{u}{\lambda}\r|^{p(x)}+\mu(x)\l|\frac{u}{\lambda}\r|^{q(x)}\Biggl]\,\diff x\le1\Biggl\},
\end{align*}
where the modular $\hat{\rho}_\mathcal{H}$ is given by
\begin{align*}
    \hat{\rho}_\mathcal{H}(u) =\into \l(|\nabla u|^{p(x)}+\mu(x)|\nabla u|^{q(x)}\r)\,\diff x+\into \l(|u|^{p(x)}+\mu(x)|u|^{q(x)}\r)\,\diff x
\end{align*}
for $u \in \Wp{\mathcal{H}}$.

The next proposition can be found in \cite[Proposition 2.14]{CGHW}.

\begin{proposition}
    Let hypotheses \textnormal{(H0)} be satisfied. Then the following holds:
    \begin{enumerate}
        \item[\textnormal{(i)}]
            If $y\neq 0$, then $\|y\|_{\hat{\rho}_\mathcal{H}}=\lambda$ if and only if $ \hat{\rho}_{\mathcal{H}}(\frac{y}{\lambda})=1$;
        \item[\textnormal{(ii)}]
            $\|y\|_{\hat{\rho}_\mathcal{H}}<1$ (resp.\,$>1$, $=1$) if and only if $ \hat{\rho}_{\mathcal{H}}(y)<1$ (resp.\,$>1$, $=1$);
        \item[\textnormal{(iii)}]
            If $\|y\|_{\hat{\rho}_\mathcal{H}}<1$, then $\|y\|_{\hat{\rho}_\mathcal{H}}^{q_+}\leqslant \hat{\rho}_{\mathcal{H}}(y)\leqslant\|y\|_{\hat{\rho}_\mathcal{H}}^{p_-}$;
        \item[\textnormal{(iv)}]
            If $\|y\|_{\hat{\rho}_\mathcal{H}}>1$, then $\|y\|_{\hat{\rho}_\mathcal{H}}^{p_-}\leqslant \hat{\rho}_{\mathcal{H}}(y)\leqslant\|y\|_{\hat{\rho}_\mathcal{H}}^{q_+}$;
        \item[\textnormal{(v)}]
            $\|y\|_{\hat{\rho}_\mathcal{H}}\to 0$ if and only if $ \hat{\rho}_{\mathcal{H}}(y)\to 0$;
        \item[\textnormal{(vi)}]
            $\|y\|_{\hat{\rho}_\mathcal{H}}\to +\infty$ if and only if $ \hat{\rho}_{\mathcal{H}}(y)\to +\infty$.
        \item[\textnormal{(vii)}]
            $\|y\|_{\hat{\rho}_\mathcal{H}}\to 1$ if and only if $ \hat{\rho}_{\mathcal{H}}(y)\to 1$.
        \item[\textnormal{(viii)}]
            If $u_n \to u$ in $\Wp{\mathcal{H}}$, then $\hat{\rho}_\mathcal{H} (u_n) \to \hat{\rho}_\mathcal{H} (u)$.
    \end{enumerate}
\end{proposition}

It turns out that $\|\cdot\|_{\hat{\rho}_\mathcal{H}}$ is a uniformly convex norm on $\WH$ and satisfies the Radon-Riesz (or Kadec-Klee) property with respect to the modular, see \cite[Propositions 2.15 and 2.19]{CGHW}.

\begin{proposition}
    Let hypotheses \textnormal{(H0)} be satisfied.
    \begin{enumerate}
        \item[\textnormal{(i)}]
            The norm $\|  \cdot \|_{ \hat{\rho}_{\mathcal{H}} }$ on $\Wp{\mathcal{H}}$ is uniformly convex.
        \item[\textnormal{(ii)}]
            For any sequence $\{u_n\}_{n \in \N} \subseteq \Wp{\mathcal{H}}$ such that
            \begin{equation*}
                u_n \weak u \quad \text{in }\Wp{\mathcal{H}}\quad\text{and}\quad
                \hat{\rho}_{\mathcal{H}} (u_n) \to \hat{\rho}_{\mathcal{H}} (u)
            \end{equation*}
            it holds that $u_n \to u$ in $\Wp{\mathcal{H}}$.
        \item[\textnormal{(iii)}]
            The norm $\|  \cdot \|_{ 1, \mathcal{H}, 0 }$ on $\Wpzero{\mathcal{H}}$ is uniformly convex.
        \item[\textnormal{(iv)}]
            For any sequence $\{u_n\}_{n \in \N} \subseteq \Wpzero{\mathcal{H}}$ such that
            \begin{equation*}
                u_n \weak u \quad \text{in }\Wpzero{\mathcal{H}}\quad\text{and}\quad
                \rho_{\mathcal{H}} ( \nabla u_n ) \to \rho_{\mathcal{H}} ( \nabla u )
            \end{equation*}
            it holds that $u_n \to u$ in $\Wpzero{\mathcal{H}}$.
    \end{enumerate}
\end{proposition}

Now we introduce the seminormed space
\begin{align*}
    L^{q(\cdot)}_\mu(\Omega)=\left \{u\in M(\Omega) \,:\, \into \mu(x) | u|^{q(x)} \,\diff x< +\infty \right \}
\end{align*}
and endow it with the seminorm
\begin{align*}
    \|u\|_{q(\cdot),\mu} =\inf \left \{ \tau >0 \,:\, \into \mu(x) \l(\frac{|u|}{\tau}\r)^{q(x)} \,\diff x  \leq 1  \right \}.
\end{align*}

The following embeddings are stated in \cite[Proposition 2.16]{CGHW}.

\begin{proposition}
    \label{proposition_embeddings}
    Let hypotheses \textnormal{(H0)} be satisfied and let $p^*(\cdot)$, $p_*(\cdot)$ given in \eqref{critical_exponent} be the critical exponents to $p(\cdot)$.
    \begin{enumerate}
        \item[\textnormal{(i)}]
            $\Lp{\mathcal{H}} \hookrightarrow \Lp{r(\cdot)}$, $\Wp{\mathcal{H}}\hookrightarrow \Wp{r(\cdot)}$, $\Wpzero{\mathcal{H}}\hookrightarrow \Wpzero{r(\cdot)}$ are continuous for all $r\in C(\close)$ with $1\leq r(x)\leq p(x)$ for all $x \in \Omega$;
        \item[\textnormal{(ii)}]
            $\Wp{\mathcal{H}} \hookrightarrow \Lp{r(\cdot)}$ is compact for $r \in C(\close) $ with $ 1 \leq r(x) < p^*(x)$ for all $x\in \close$;
        \item[\textnormal{(iii)}]
            The trace operator $i_{r(\cdot)}: \Wp{\mathcal{H}} \hookrightarrow \Lprand{r(\cdot)}$ is compact for $r \in C(\close) $ with $ 1 \leq r(x) < p_*(x)$ for all $x\in \close$;
        \item[\textnormal{(iv)}]
            $\Lp{\mathcal{H}} \hookrightarrow L^{q(\cdot)}_\mu(\Omega)$ is continuous;
        \item[\textnormal{(v)}]
            $L^{q(\cdot)}(\Omega) \hookrightarrow \Lp{\mathcal{H}}$ is continuous;
        \item[\textnormal{(vi)}]
            $\Wp{\mathcal{H}}\hookrightarrow \Lp{\mathcal{H}}$ is compact.
    \end{enumerate}
\end{proposition}

For any $s \in \R$ we denote $s^\pm = \max \{ \pm s, 0 \}$, that means $s = s^+ - s^-$ and $|s| = s^+ + s^-$. For any function $v \colon \Omega \to \R$ we denote $v^\pm (\cdot) = [v(\cdot)]^\pm$. The spaces $\WH$ and $\Wpzero{\mathcal{H}}$ are closed under $\max$ and $\min$, see \cite[Proposition 2.17]{CGHW}.

\begin{proposition}
    Let hypotheses \textnormal{(H0)} be satisfied.
    \begin{enumerate}
        \item[\textnormal{(i)}]
            if $u \in \Wp{\mathcal{H}}$, then $\pm u^\pm \in \Wp{\mathcal{H}}$ with $\nabla (\pm u ^\pm) = \nabla u 1_{\{\pm u > 0\}}$;
        \item[\textnormal{(ii)}]
            if $u_n \to u$ in $\Wp{\mathcal{H}}$, then $\pm u_n^\pm \to \pm u^\pm$ in $\Wp{\mathcal{H}}$;
        \item[\textnormal{(iii)}]
            if $u \in \Wpzero{\mathcal{H}}$, then $\pm u^\pm \in \Wpzero{\mathcal{H}}$.
    \end{enumerate}
\end{proposition}

Let $X=\WH$ or $X=\Wpzero{\mathcal{H}}$ and let $A\colon X\to X^*$ be the nonlinear operator defined by
\begin{align}
    \label{defA}
    \langle A(u),v\rangle :=
    \into \l(|\nabla u|^{p(x)-2}\nabla u+\mu(x)|\nabla u|^{q(x)-2}\nabla u \r)\cdot\nabla v \,\diff x
\end{align}
for $u,v\in X$ with $\langle \cdot,\cdot\rangle$ being the duality pairing
between $X$ and its dual space $X^*$. The following proposition summarizes the main properties of $A\colon X\to X^*$, see \cite[Theorem 3.3 and Proposition 3.4]{CGHW}.

\begin{proposition}
    \label{prop1}
    Let hypotheses \textnormal{(H0)} be satisfied. Then, the operator $A$ defined by \eqref{defA} is bounded, continuous, strictly monotone and of type $(\Ss_+)$, that is,
    \begin{align*}
        u_n\weak u \quad \text{in }X \quad\text{and}\quad  \limsup_{n\to\infty}\,\langle Au_n,u_n-u\rangle\le 0,
    \end{align*}
    imply $u_n\to u$ in $X$.
\end{proposition}

In what follows, to shorten notation, we write $\|\cdot\|=\|\cdot\|_{\hat{\rho}_\mathcal{H}}$ for the norm in $\WH$ and $\|\cdot\|_0=\|\cdot\|_{1,\mathcal{H},0}$ for the norm in $W^{1,\mathcal{H}}_0(\Omega)$. The corresponding dual spaces are denoted by $\WH^*$ and $W^{1,\mathcal{H}}_0(\Omega)^*$, respectively.  Given a Banach space $X$ and its dual space $X^*$ we denote
\begin{align*}
    \K(X^*) = \l\{ P\subset X^* \,:\, P \neq \emptyset, \, P \text{ is closed and convex}\r\}.
\end{align*}

Let $X$ be a real Banach space with its dual space $X^*$. A function $J\colon X\to \R$ is said to be locally Lipschitz at $u\in X$ if there
exist a neighborhood $N(u)$ of $u$ and a constant $L_u>0$ such that
\begin{align*}
    |J(w)-J(v)|\leq L_u\|w-v\|_X \quad \text{for all } w, v\in N(u).
\end{align*}

\begin{definition}
    Let $J \colon X \to \R$ be a locally Lipschitz function and
    let $u,v\in X$. The generalized directional derivative $J^\circ(u; v)$ of $J$ at the point $u$ in the direction $v$ is defined by
    \begin{align*}
        J^\circ(u;v): = \limsup \limits_{w\to u,\, t\downarrow 0 } \frac{J(w+t v)-J(w)}{t}.
    \end{align*}
    The generalized gradient $\partial J\colon X\to 2^{X^*}$ of $J
        \colon X \to \R$ is defined by
    \begin{align*}
        \partial J(u): =\left \{\, \xi\in X^{*} \,:\, J^\circ (u; v)\geq \langle\xi, v\rangle_{X^*\times X} \ \text{ for all } v \in    X  \right \} \quad \text{for all } u\in X.
    \end{align*}
\end{definition}

The next proposition collects some basic results, see \cite{Cla90} or \cite{Papageorgiou-Winkert-2018}.

\begin{proposition}
    \label{P1}
    Let $J\colon X \to \R$ be locally Lipschitz with Lipschitz constant
    $L_{u}>0$ at $u\in X$. Then we have the following:
    \begin{enumerate}
        \item[{\rm(i)}]
            The function $v\mapsto J^\circ(u;v)$ is positively    homogeneous, subadditive, and satisfies
            \begin{align*}
                |J^\circ(u;v)|\leq L_{u}\|v\|_X \quad \text{for all }v\in X.
            \end{align*}
        \item[{\rm(ii)}]
            The function $(u,v)\mapsto J^\circ(u;v)$ is upper semicontinuous.
        \item[{\rm(iii)}]
            For each $u\in X$, $\partial J(u)$ is a nonempty, convex,
            and weak$^*$ compact subset of $X^*$ with $ \|\xi\|_{X^{*}}\leq L_{u}$ for all $\xi\in\partial J(u)$.
        \item[{\rm(iv)}]
            $J^\circ(u;v) = \max\left\{\langle\xi,v\rangle_{X^*\times X} \mid \xi\in\partial J(u)\right\}$ for all $v\in X$.
        \item[{\rm(v)}]
            The multi-valued function $X\ni u\mapsto \partial J(u)\subset X^*$ is upper semicontinuous from $X$ into w$^*$-$X^*$.
    \end{enumerate}
\end{proposition}

Assume $p_1, p_2\in C(\close)$ and $(p_{j})_-  \geq 1$, $(j=1,2)$.  Let $F$ be a function from $\Omega\times \R$ into $2^\R$.  For each measurable $u \colon \Omega\to \R$, we consider the function
$F(u) \colon \Omega \to 2^\R$, $F(u)(x) = F(x,u(x))$ and denote
$\tilde{F}(u) = \{ v \in M(\Omega) \,:\, v(x) \in F(x,u(x))$ for a.\,a. $x\in \Omega\}$. The following theorem can be found in \cite[Theorem 7.3]{carl:mvi21}.

\begin{theorem}
    \label{t*31}
    Assume $F\colon\Omega\times\R\to 2^{\R}$ satisfies the following conditions:
    \begin{enumerate}
        \item[\textnormal{(i)}]
            For a.\,a.\,$x\in \Omega$ and for all $u\in \R$, $F(x,u)$ is closed and nonempty;
        \item[\textnormal{(ii)}]
            $F$ is graph measurable;
        \item[\textnormal{(iii)}]
            For a.\,a.\,$x\in \Omega$, the function $u\mapsto F(x,u)$ is Hausdorff-upper semicontinuous (h-u.s.c.\,for short);
        \item[\textnormal{(iv)}]
            There exist $a\in L^{p_2(\cdot)}(\Omega)$ and $b >0$ such that
            \begin{align*}
                |v| \leq a(x) + b|u|^{\frac{p_1(x)}{p_2(x)}}
            \end{align*}
            for a.\,a.\,$x\in \Omega$ and for all $v\in F(x,u)$.
    \end{enumerate}
    Thus, for each $u\in L^{p_1(\cdot)}(\Om)$, $\tilde{F}(u)$ is a (nonempty) closed subset of $L^{p_2(\cdot)}(\Om)$ and the mapping $\tilde{F}\colon u\mapsto \tilde{F}(u)$ is h-u.s.c.\,from $L^{p_1(\cdot)}(\Om)$ to $2^{L^{p_2(\cdot)}(\Om)}$.
\end{theorem}

\begin{remark}
    We have an analogous result to Theorem \ref{t*31}, where $\Omega$ is replaced by $\Gamma$.  In fact, a straightforward generalization of Theorem \ref{t*31} holds true with $\Om$ being a measure space on which Lebesgue and Sobolev spaces with variable exponents are defined.
\end{remark}

The following theorem was proved in \cite[Theorem 2.2]{le:ret11}. We use the notation $B_R(0):=\{u\in X \ :  \ \|u\|_X<R\}$.

\begin{theorem}
    \label{surjectivitytheorem}
    Let $X$ be a real reflexive Banach space,  let $F\colon D(F)\subset X\to 2^{X^*}$ be a maximal monotone operator, let $G\colon D(G)=X\to 2^{X^*}$ be a bounded multi-valued pseudomonotone operator and let $L\in X^*$.
    Assume that there exist $u_0\in X$ and $R\ge \|u_0\|_X$ such that $D(F)\cap B_R(0)\neq\emptyset$ and
    \begin{align*}
        \langle\xi+\eta-L,u-u_0\rangle_{X^*\times X}>0
    \end{align*}
    for all $u\in D(F)$ with $\|u\|_X=R$, for all $\xi\in F(u)$ and for all $\eta\in G(u)$. Then the inclusion
    \begin{align*}
        F(u)+G(u)\ni L
    \end{align*}
    has a solution in $D(F)$.
\end{theorem}

An important tool in extending our results to discontinuous Nemytskij operators is  the next fixed point result, see \cite[Proposition 2.39]{Carl-Heikkilae-2011} or \cite[Theorem 1.1.1]{Carl-Heikkilae-2000}.

\begin{theorem}
    \label{theorem-fixed-point}
    Let $P$ be a subset of an ordered normed space $X$, and let $G\colon P\to P$ be an increasing mapping, that is, $x,y\in P$ with $x\leq y$ implies $Gx\leq Gy$. Then the following holds true:
    \begin{enumerate}
        \item[\textnormal{(i)}]
            If the image $G(P)$ has a lower bound in $P$ and increasing sequences of $G(P)$ converge weakly in $P$, then $G$ has the smallest fixed point $x_*$ given by
            $x_*=\min\{x\,:\,Gx \leq x\}$.
        \item[\textnormal{(ii)}]
            If the image $G(P)$ has an upper bound in P and decreasing sequences of $G(P)$ converge weakly in $P$, then $G$ has the greatest fixed point $x^*$ given by $x^*=\max\{x\,:\,x\leq Gx\}$.
    \end{enumerate}
\end{theorem}

\section{Coercive Case: Proof of Theorem \ref{T101}}\label{section_3}

In this section, we are going to prove Theorem \ref{T101}. First, recall that the embedding $i_{r_1(\cdot)} \colon \WH \to  L^{r_1(\cdot)}(\Omega)$, $u\mapsto u$, and the trace operator
$i_{r_2(\cdot)} \colon \WH \to  L^{r_2(\cdot)}(\Gamma)$, $u\mapsto u|_{\Gamma}$, are compact due to (F2) and Proposition \ref{proposition_embeddings}\textnormal{(ii)}, \textnormal{(iii)}.
Let $i_{r_1(\cdot)}^* \colon L^{r_1'(\cdot)}(\Omega) \to \WH^*$, and $i_{r_2(\cdot)}^* \colon L^{r_2'(\cdot)}(\Gamma) \to \WH^*$  be their adjoints.  As a consequence of (F1), for any $u\in M(\Om)$, the set of measurable selections of $f(\cdot , u)$,
\begin{align*}
    \tilde{f}(u) = \l\{ \eta \in M(\Omega) \,:\, \eta(x)\in f(x,u(x)) \text{ for a.\,a.\,} x\in \Omega\r\},
\end{align*}
is nonempty. Similarly, for any $u\in M(\Gamma)$, the set of measurable selections of $f_{\Gamma}(\cdot , u)$,
\begin{align*}
    \tilde{f}_\Gamma(u) = \l\{ \eta \in M(\Gamma) \,:\, \eta(x)\in f_\Gamma (x,u(x)) \text{ for a.\,a.\,} x\in \Gamma\r\},
\end{align*}
is also nonempty.

Moreover, from  (F2), $\tilde{f}(u)\subset L^{r_1'(\cdot)}(\Om)$ if $u\in L^{r_1(\cdot)}(\Omega)$ and
$\tilde{f}_\Gamma(u)\subset L^{r_2'(\cdot)}(\Gamma)$ if $u\in L^{r_2(\cdot)}(\Gamma)$.
Let us consider the mappings $\tilde{f} \colon L^{r_1(\cdot)}(\Omega) \to L^{r_1'(\cdot)}(\Omega)$, $u\mapsto \tilde{f}(u)$ and $\F = i_{r_1(\cdot)}^* \tilde{f} i_{r_1(\cdot)} \colon \WH \to 2^{\WH^*}$, that is, $\F(u) = \{ \hat{\eta}\in \WH^* \,:\, \eta\in\tilde{f}(u)\}$, where $\hat{\eta}\in \WH^*$ is defined for each $\eta\in L^{r_1'(\cdot)}(\Om)$ by
\begin{align*}
    \la \hat{\eta} , v\ra = \int_\Om\eta v \,\diff x \quad\text{for all } v\in \WH.
\end{align*}
Similarly, we define
\begin{align*}
    &\tilde{f_\Gamma } \colon L^{r_2(\cdot)}(\Gamma ) \to L^{r_2'(\cdot)}(\Gamma ), \quad u\mapsto \tilde{f_\Gamma }(u),\\
    &{\F}_\Gamma  = i_{r_2(\cdot)}^* \tilde{f_\Gamma } i_{r_2(\cdot)} \colon \WH \to 2^{\WH^*}
\end{align*}
with ${\F}_\Gamma (u) = \{ \hat{\eta}\in \WH^* \,:\, \eta\in\tilde{f}_\Gamma (u)\}$, where $\hat{\eta}\in \WH^*$ is defined for each $\eta\in L^{r_2'(\cdot)}(\Gamma )$ by
\begin{align*}
    \la \hat{\eta} , v\ra = \int_\Gamma \eta v \,\diff \sigma \quad\text{for all } v\in \WH.
\end{align*}

We have the following crucial property of $\F$ and ${\F}_\Gamma$.

\begin{proposition}
    \label{p35}
    Let hypotheses \textnormal{(H0)}, \textnormal{(F1)} and \textnormal{(F2)} be satisfied. The mappings $\F = i_{r_1(\cdot)}^* \tilde{f} i_{r_1(\cdot)} $ and  ${\F}_\Gamma  = i_{r_2(\cdot)}^* \tilde{f_\Gamma } i_{r_2(\cdot)}$ are pseudomonotone and bounded from $\WH$ into $\K(\WH^*)$.
\end{proposition}

\begin{proof}
    First, let us note that $\tilde{f}(u)\in \K(L^{r_1'(\cdot)}(\Omega))$ for all $u\in L^{r_1(\cdot)}(\Omega)$.  In fact, the convexity of $\tilde{f}(u)$ and the boundedness of $\tilde{f}$ (as a multi-valued mapping) follow directly from (F1) and (F2).  The proof of the closedness of $\tilde{f}(u)$ is a direct consequence of the fact that $f(x,t)$ is a closed bounded interval in $\R$ for a.\,a.\,$x\in \Omega$ and for all $t\in \R$.

    Next, we show that the graph of $\F$ is (sequentially) weakly closed in $\WH \times \WH^*$.   Assume that $\{u_n\}_{n\in\N}$ and $\{u_n^*\}_{n\in\N}$ are sequences in $\WH$ and $\WH^*$, respectively, such that
    \begin{align}
        u_n \rh u \quad        & \text{in } \WH,\label{36}          \\
        u_n^* \rh u^* \quad    & \text{in } \WH^*,\label{37}        \\
        u_n^*\in \F(u_n) \quad & \text{for all }n\in \N. \label{38}
    \end{align}
    Let us prove that
    \begin{equation}
        \label{39}
        u^*\in \F(u) .
    \end{equation}
    Since $u_n^*\in i_{r_1(\cdot)}^* \tilde{f} i_{r_1(\cdot)} (u_n)$, there exists $\tilde{u}_n\in \tilde{f}(i_{r_1(\cdot)} (u_n)) = \tilde{f}(u_n)$ such that $u_n^* = i_{r_1(\cdot)}^* (\tilde{u}_n) = \tilde{u}_n|_{\WH}$. It follows from \eqref{36} and the compactness of the embedding $i_{r_1(\cdot)}$ that
    \begin{equation}
        \label{39a}
        u_n \to u \quad \text{in } L^{r_1(\cdot)}(\Om) .
    \end{equation}
    Hence, from Theorem \ref{t*31} and the growth condition in hypothesis (F2) it follows that $h^*(\tilde{f}(u_n), \tilde{f}(u))\to 0$, and thus
    \begin{equation*}
        \inf_{w^*\in \tilde{f}(u)} \l\|\tilde{u}_n - w^* \r\|_{r_1'(\cdot)} \to 0 .
    \end{equation*}
    Consequently, there is a sequence $\{w^*_n\}_{n\in\N}\subset \tilde{f}(u)$ such that $\|\tilde{u}_n - w_n^* \|_{r_1'(\cdot)} \to 0$. Since $\tilde{f}(u)$ is bounded in $L^{r_1'(\cdot)}(\Omega)$, by passing to  a subsequence if necessary, we can assume that $w^*_n \rh w^*_0$ in $L^{{r_1'(\cdot)}}(\Omega)$ for some $w^*_0\in L^{r_1'(\cdot)}(\Omega)$.  Moreover, $w^*_0\in \tilde{f}(u)$ by the convexity and closedness of $\tilde{f}(u)$.  We have
    \begin{equation}
        \label{40a}
        \tilde{u}_n \rh w^*_0 \;\quad \text{in } L^{r_1'(\cdot)}(\Omega) ,
    \end{equation}
    and from the compactness of $i^*_{r_1(\cdot)}$,
    \begin{align*}
        u^*_n = i^*_{r_1(\cdot)} (\tilde{u}_n) \to  i^*_{r_1(\cdot)} (w^*_0) = w^*_0|_{\WH} \quad \text{in } \WH^* .
    \end{align*}
    From \eqref{37}, we obtain $u^* =  i^*_{r_1(\cdot)} (w^*_0) \in  i^*_{r_1(\cdot)} \tilde{f}(u) = \F (u)$.  Hence, \eqref{39} is proved.

    As a direct consequence of this property, we see that $\F(u)$ is closed in $X^*$. Furthermore, from the growth condition in (F2), we see that $\tilde{f}$ is  bounded  from $L^{r_1(\cdot)}(\Omega)$ into $2^{L^{r_1'(\cdot)}(\Omega)}$ and thus $\F$ is a bounded mapping from $\WH$ into $\K(\WH^*)$. Therefore, to prove its pseudomonotonicity, we only need to check that $\F$ is generalized pseudomonotone. For this purpose, let $\{u_n\}_{n\in\N}$ and  $\{u_n^*\}_{n\in\N}$ be sequences satisfying \eqref{36}--\eqref{38} and let  $\{\tilde{u}_n\}_{n\in\N}$ as well as $w_0^*$ be as above.  We have, from \eqref{39a} and \eqref{40a},
    \begin{align*}
        \la u_n^* , u_n \ra  =  \la i_{r_1(\cdot)}^* (\tilde{u}_n) , u_n \ra
        & =   \la \tilde{u}_n ,  u_n \ra_{L^{r_1'(\cdot)}(\Om) , L^{r_1(\cdot)}(\Om)}\\
        &  \to \la w_0^* ,  u \ra_{L^{r_1'(\cdot)}(\Om) , L^{r_1(\cdot)}(\Om)}\\
        &=  \la  i_{r_1(\cdot)}^* (w_0^* ) ,  u \ra
        =  \la  u^*  ,  u \ra .
    \end{align*}
    This limit shows that $\F$ is generalized pseudomonotone and thus pseudomonotone.  It also follows from the arguments above that $\F$ is bounded. The proof of the pseudomonotonicity and boundedness of ${\F}_{\Gamma}$ follows  similar arguments.
\end{proof}

We are now ready to prove Theorem \ref{T101}.

\begin{proof}[Proof of Theorem \ref{T101}]
    We are going to apply Theorem \ref{surjectivitytheorem}. Since $A$ is continuous, strictly monotone and bounded on $\WH$ with domain $D(A) = \WH$, it is a (single-valued) bounded and pseudomonotone mapping from $\WH$ to $2^{\WH^*}$. It follows from Proposition \ref{p35} that  $A + \F +  \F_\Gamma$ is a pseudomonotone and bounded mapping  from $\WH$ into $2^{\WH^*}$.

    We note that $ \pa I_K$ is a maximal monotone mapping from $\WH$ to $2^{\WH^*}$ with domain $D( \pa I_K) = K$. According to Theorem \ref{surjectivitytheorem}, under the coercivity condition \eqref{101}, problem \eqref{problem} has at least one solution.
\end{proof}

A straightforward consequence of Theorem \ref{T101} is the following result.

\begin{corollary}
    \label{C101}
    Let hypotheses \textnormal{(H0)}, \textnormal{(F1)} and \textnormal{(F2)} be satisfied and suppose that for fixed $u_0\in K$ the following coercivity condition holds
    \begin{align}
        \label{101}
        \lim_{\substack{\|u\|_{1,\mathcal{H}} \to \infty \\ u\in K}} \left[ \inf_{\substack{\eta^* \in \mathcal{F}(u)\\ \zeta^* \in \mathcal{F}_\Gamma (u)}}{\l\langle Au+\eta^*+\zeta^*,u-u_0\r\rangle }\right]=\infty.
    \end{align}
    Then problem \eqref{problem} has at least one solution.
\end{corollary}

\section{Noncoercive Case: Proof of Theorem \ref{T102}}\label{section_4}

In order to prove Theorems \ref{T102} and \ref{T103}, let us first establish the following general existence and enclosure theorem for solutions of \eqref{problem} if a finite number of sub- and supersolutions exist and $f$ has a local growth between those sub- and supersolutions.

\begin{theorem}
    \label{T4-1}
    Let hypotheses \textnormal{(H0)} and \textnormal{(F1)} be satisfied and let $\un{u}_i$  $(i=1,\dots , k)$ be subsolutions and $\ov{u}_j$ $(j=1,\dots , m)$ be supersolutions of \eqref{problem} such that
    \begin{equation*}
        \un{u} = \max\l\{ \un{u}_i \,:\, i = 1,\dots , k\r\} \le \ov{u} = \min\l\{\ov{u}_j \,:\, j = 1,\dots , m\r\} \quad\text{a.\,e.\,in } \Omega,
    \end{equation*}
    and $\un{u}_i \vee K \subset K$  for all $i\in \{1,\dots , k\}$ and $\ov{u}_j \wedge K \subset K$ for all $j\in \{1,\dots , m\}$.

    Suppose there exist $\tau_1\in C(\close)$, $\tau_2\in C(\Gamma)$,  $1<\tau_1(x)<p^*(x)$ for all $x\in\close$, $1<\tau_2(x)<p_*(x)$ for all $x\in\Gamma$ such that
    \begin{align}
        \label{g1}
        \begin{split}
            \sup \l\{|\eta| \,:\, \eta \in f(x,s)\r\} &\leq k_\Omega(x)\quad\text{for a.\,a.\,}x\in\Omega,\\
            \sup \l\{|\zeta| \,:\, \zeta \in f_\Gamma(x,s)\r\} &\leq k_\Gamma(x)\quad\text{for a.\,a.\,}x\in\Gamma,
        \end{split}
    \end{align}
    for all $s\in [\underline{u}(x),\overline{u}(x)]$ and for some $k_\Omega\in L^{\tau_1'(\cdot)}(\Omega)$, $k_\Gamma \in L^{\tau_2'(\cdot)}(\Gamma)$.

    Then, there exists a solution $u$ of \eqref{problem} such that
    \begin{align*}
        \un{u} \leq u \le \overline{u} \quad \text{a.\,e.\,in } \Omega.
    \end{align*}
\end{theorem}

\begin{proof}
    First, note that by increasing $\tau_1$ and $\tau_2$ in Definitions \ref{D102} and \ref{D103} for each $\un{u}_i$ and $\ov{u}_j$ and in the growth condition \eqref{g1} appropriately, to simplify the notation we can assume without loss of  generality that the functions $\tau_1$ and $\tau_2$ in the definitions of $\un{u}_i$ $(1\le i \le k)$ and $\ov{u}_j$ $(1\le j \le m)$ in Definitions \ref{D102} and \ref{D103} and in the growth condition \ref{g1}, are the same.

    For $i\in\{1,\dots , k\}$ and $j\in \{1,\dots , m\}$, let $\un{\xi}_i$, $\un{\eta}_i$, and $\ov{\xi}_j$, $\ov{\eta}_j$ be the functions associated with $\un{u}_i$ and $\ov{u}_j$ as in Definitions \ref{D102} and \ref{D103}. We define the truncation function $f_0$ of $f$ as follows: Let
    \begin{align*}
        \Omega_1 = \l\{ x\in \Omega \,:\, \un{u}(x) = \un{u}_1(x)\r\}, \quad \Omega^1 = \l\{ x\in \Om \,:\, \ov{u}(x) = \ov{u}_1(x)\r\},
    \end{align*}
    and define
    \begin{align*}
        \Omega_i &= \left\{ x\in \Omega\setminus {\textstyle \displaystyle \bigcup_{l=1}^{i-1}}\, \Omega_l \,:\, \un{u}(x) = \un{u}_i (x) \right\},\\
        \Omega^j &= \left\{ x\in \Omega\setminus{\textstyle \displaystyle \bigcup_{l=1}^{j-1}}\, \Omega^l \,:\, \ov{u}(x) = \ov{u}^j (x) \right\}
    \end{align*}
    for all $i = 2,\dots, k$ and for all $j = 2,\dots, m$. Next, we define
    \begin{align*}
        \un{\eta} = \sum_{i=1}^k \un{\eta}_i \chi_{\Om_i}
        \quad \text{and}\quad
        \overline{\eta} = \sum_{j=1}^m \overline{\eta}_j \chi_{\Omega^j} ,
    \end{align*}
    where $\chi_A$ is the characteristic function of $A \subset \Omega$.  From their definitions, we have $\un{\eta}, \ov{\eta}\in L^{\tau_1' (\cdot)}(\Om)$ and furthermore, $\un{\eta}(x) \in f(x,\un{u}(x))$ and $\ov{\eta}(x) \in f(x,\ov{u}(x)) $ for a.\,a.\,$ x\in \Omega$.

    Let $f_0 \colon \Omega\times \R \to 2^\R$ be defined by
    \begin{align}
        \label{truncated1}
        f_0(x,u) =
        \begin{cases}
            \{\un{\eta}(x)\} & \text{if } u < \un{u}(x)                 \\[1ex]
            f(x,u)           & \text{if } \un{u}(x) \le u \le \ov{u}(x) \\[1ex]
            \{\ov{\eta}(x)\} & \text{if } u > \ov{u}(x).
        \end{cases}
    \end{align}

    Similarly, let
    \begin{align*}
        \Gamma_1 = \l\{ x\in \Gamma \,:\, \un{u}(x) = \un{u}_1(x)\r\}, \quad \Gamma^1 = \l\{ x\in \Gamma \,:\, \ov{u}(x) = \ov{u}_1(x)\r\},
    \end{align*}
    and
    \begin{align*}
        \Gamma_i &= \left\{ x\in \Gamma\setminus {\textstyle \displaystyle\bigcup_{l=1}^{i-1}}\, \Gamma_l \,:\, \un{u}(x) = \un{u}_i (x) \right\},\\
        \Gamma^j &= \left\{ x\in \Gamma\setminus {\textstyle \displaystyle\bigcup_{l=1}^{j-1}}\, \Gamma^l \,:\, \ov{u}(x) = \ov{u}^j (x) \right\}
    \end{align*}
    for all $i = 2,\dots, k$ and for all $j = 2,\dots, m$. We define
    \begin{align*}
        \un{\zeta} = \sum_{i=1}^k \un{\zeta}_i \chi_{\Gamma_i}
        \quad\text{and}\quad
        \ov{\zeta} = \sum_{j=1}^m \ov{\zeta}_j \chi_{\Gamma^j} ,
    \end{align*}
    where, as above, $\chi_A$ is the characteristic function of $A \subset \Gamma$.  We also have $\un{\zeta}, \ov{\zeta}\in L^{\tau_2' (\cdot)}(\Gamma)$ and $\un{\zeta}(x) \in f_\Gamma (x,\un{u}(x))$ and $\ov{\zeta}(x) \in f_\Gamma(x,\ov{u}(x)) $ for a.\,a.\,$ x\in \Gamma$.

    Let $f_{0\Gamma} \colon \Gamma\times \R \to 2^\R$ be defined by
    \begin{align}
        \label{truncated2}
        f_{0\Gamma}(x,u) =
        \begin{cases}
            \{\un{\zeta}(x)\} & \text{if }  u < \un{u}(x)                 \\[1ex]
            f_\Gamma(x,u)     & \text{if }  \un{u}(x) \le u \le \ov{u}(x) \\[1ex]
            \{\ov{\zeta}(x)\} & \text{if }  u > \ov{u}(x).
        \end{cases}
    \end{align}
    Then,  $f_0$ and $f_{0\Gamma}$ given in \eqref{truncated1} and \eqref{truncated2}, respectively, satisfy (F1).  Moreover, it follows from (\ref{g1}) and the definitions of $f_0$ and $f_{0\Gamma}$ that
    \begin{align}
        \label{4-3}
        \begin{split}
            &\sup\{|v| \,:\, v\in f_0(x,u)\} \le k_\Om(x) + |\un{\eta}(x)| + |\ov{\eta}(x)|\quad \text{for a.\,a.\,} x\in \Omega,\\
            &\sup\{|v| \,:\, v\in f_{0\Gamma}(x,u)\} \le k_\Gamma(x) + |\un{\zeta}(x)| + |\ov{\zeta}(x)|
            \quad \text{for a.\,a.\,}x\in \Gamma,
        \end{split}
    \end{align}
    for all $u\in\R$, where $k_\Om +|\un{\eta}|+|\ov{\eta}|\in L^{\tau_1'(\cdot)}(\Om)$ and $k_\Gamma +|\un{\zeta}|+|\ov{\zeta}|\in L^{\tau_2'(\cdot)}(\Gamma)$.

    In particular, $f_0$ and $f_{0\Gamma}$ satisfy (F2) with $\beta=\beta_\Gamma=0$ and $\al = k_\Omega+|\un{\eta}|+|\ov{\eta}|$, $\al_\Gamma = k_\Gamma+|\un{\zeta}|+|\ov{\zeta}|$.  It follows from Proposition \ref{p35} that the mappings $ i_{\tau_1(\cdot)}^* \tilde{f}_0  i_{\tau_1(\cdot)}$ and $ i_{\tau_2(\cdot)}^* \tilde{f}_{0\Ga}  i_{\tau_2(\cdot)}$ are bounded and pseudomonotone from $\WH$ to $\K(\WH^*)$.

    Next, let us define a truncation-regularization function $b$ as follows.  For $x\in\Om$ and $u\in \R$, let
    \begin{align}
        \label{4-4}
        b(x,u) =
        \begin{cases}[u-\ov{u}(x)]^{q(x)-1}  & \text{if } u > \ov{u}(x)                  \\[1ex]
             0                       & \text{if } \un{u}(x) \le  u \le \ov{u}(x) \\[1ex]
             -[\un{u}(x)-u]^{q(x)-1} & \text{if } u < \un{u}(x).
        \end{cases}
    \end{align}

    Here in what follows, we denote by $C$ a generic positive constant that may change from line to line. Since $\un{u},\ov{u}\in W^{1,\Hi}(\Om)$ and $W^{1,\Hi}(\Om) \emb L^{\q}(\Om)$, we see that
    \begin{equation}
        \label{4-5}
        |b(x,u)| \le a_1(x) + C |u|^{q(x)-1},
    \end{equation}
    for a.\,a.\,$x\in \Omega$ and for all $u\in \R$, where $a_1\in L^{q'(\cdot)}(\Omega)$.

    This implies that the mapping $\B \colon L^{\q}(\Omega) \to L^{q'(\cdot)}(\Omega)$ given by
    \begin{align*}
        \langle \B(u) , v\rangle = \int_\Omega b(x,u) v \,\diff x\quad \text{for all } u,v\in L^{\q}(\Omega)
    \end{align*}
    is continuous and bounded.  Moreover, thanks to the compactness of the embedding $i_{\q} \colon \WH \emb L^{\q}(\Omega)$, the mapping $i_{\q}^* \B i_\q \colon \WH \to \WH^*$ is bounded and completely continuous.  As a consequence, the mapping  $i_{\q}^* \B i_\q$ is a (single-valued) pseudomonotone and bounded mapping from $\WH$ into $\WH^*$.

    Furthermore, there is $a_2>0$ such that
    \begin{align}
        \label{4-6}
        \begin{split}
            \l\la i_{\q}^* \B i_\q (u) , u\r\ra
            & = \l\la \B  (u) , u\r\ra_{L^{q'(\cdot)}(\Omega) , L^\q(\Omega)} \\
            & = \int_\Om b(x,u) u \,\diff x \\
            & \geq a_2 \int_\Om |u|^{q(x)} \,\diff x - C  \quad \text{for all }u\in \WH.
        \end{split}
    \end{align}

    For $i\in \{1,\dots , k\}$, $j\in \{1,\dots, m\}$, $x\in \Omega$ and $u\in \R$, we define
    \begin{align*}
        T_i(x,u) &= |\un{\eta}_i(x)  -\un{\eta}(x)|\hat{\sigma} \left( \frac{u-\un{u}_i(x)}{\un{u}(x)-\un{u}_i(x)}\right),\\
        T^j(x,u) &= |\ov{\eta}_j (x) - \ov{\eta}(x)| \left[1-\hat{\sigma}\left(\frac{u-\ov{u}(x)}{\ov{u}_j(x) - \ov{u}(x)} \right) \right],
    \end{align*}
    where
    \begin{align*}
        \hat{\sigma}(s) =
        \begin{cases}
            1,   & \text{if }s\le 0,        \\
            1-s, & \text{if }0 \le s \le 1, \\
            0,   & \text{if } s \ge 1.
        \end{cases}
    \end{align*}
    Similarly, for $i\in \{1,\dots , k\}$, $j\in \{1,\dots, m\}$, $x\in \Gamma$, and $u\in \R$, we define
    \begin{align*}
        U_i(x,u) &= |\un{\zeta}_i(x)  -\un{\zeta}(x)|\hat{\sigma} \left( \frac{u-\un{u}_i(x)}{\un{u}(x)-\un{u}_i(x)}\right),\\
        U^j(x,u) &= |\ov{\zeta}_j (x) - \ov{\zeta}(x)| \left[1-\hat{\sigma}\left(\frac{u-\ov{u}(x)}{\ov{u}_j(x) - \ov{u}(x)} \right) \right].
    \end{align*}

    Straightforward calculations show that $T_i(\cdot, u)$, $T^j(\cdot, u)\in L^{\tau_1'(\cdot)}(\Omega)$ whenever $u\in L^{\tau_1(\cdot)}(\Omega)$ and
    \begin{align*}
        0 \le T_i(x,u) \le |\un{\eta}_i(x)  -\un{\eta}(x)|
        \quad\text{and}\quad
        0 \le T^j(x,u) \le |\ov{\eta}_j - \ov{\eta}(x)|
    \end{align*}
    for a.\,a.\,$x\in\Om$ and for all $u\in \R$. It follows that $\T_i \colon u\mapsto T_i(\cdot,u)$, $\T^j \colon u\mapsto T^j(\cdot, u)$ $ (1\le i \le k, 1\le j\le m)$ are bounded and continuous operators from $L^{\tau_1(\cdot)}(\Omega)$ to $L^{\tau_1'(\cdot)}(\Omega)$.  Hence, due to the compactness of the embedding operator, $i_{\tau_1(\cdot)}^* \T_i i_{\tau_1(\cdot)}$ and $i_{\tau_1(\cdot)}^* \T^j i_{\tau_1(\cdot)}$ are completely continuous and are thus (single-valued) pseudomonotone mappings from $\WH$ into $\WH^*$.

    Analogously, $U_i(\cdot, u)$, $U^j(\cdot, u)\in L^{\tau_2'(\cdot)}(\Ga)$ whenever $ u\in L^{\tau_2(\cdot)}(\Ga)$ and
    \begin{align*}
        0 \le U_i(x,u) \le |\un{\zeta}_i(x)  -\un{\zeta}(x)|
        \quad\text{and}\quad
        0 \le T^j(x,u) \le |\ov{\zeta}_j - \ov{\zeta}(x)|
    \end{align*}
    for a.\,a.\,$x\in\Ga$ and for all $u\in \R$. Thus, $\U_i \colon u\mapsto U_i(\cdot,u)$, $\U^j \colon u\mapsto U^j(\cdot, u)$ $(1\le i \le k, 1\le j\le m)$ are bounded and continuous operators from $L^{\tau_2(\cdot)}(\Gamma)$ to $L^{\tau_2'(\cdot)}(\Gamma)$.  Therefore, by the compactness of the trace operator, $i_{\tau_2(\cdot)}^* \U_i i_{\tau_2(\cdot)}$ and $i_{\tau_2(\cdot)}^* \U^j i_{\tau_2(\cdot)}$ are  completely continuous and  are  (single-valued) pseudomonotone mappings from $\WH$ into $\WH^*$ as well.

    Let us consider the following auxiliary variational inequality:  Find $u\in  K$ and $\eta\in L^{\tau_1'(\cdot)}(\Omega)$,  $\zeta\in L^{\tau_2'(\cdot)}(\Gamma)$, such that
    \begin{align}
        \label{4-7}
        \begin{split}
            \eta(x)&\in f_0(x,u(x)) \quad \text{for a.\,a.\,} x\in \Omega,\\
            \zeta(x)&\in f_{0\Ga}(x,u(x))\quad \text{for a.\,a.\,} x\in \Ga,
        \end{split}
    \end{align}
    and
    \begin{align}
        \label{4-8}
        \begin{split}
            &\la A u, v-u\ra +\int_\Om \eta (v-u) \,\diff x + \int_\Om \zeta (v-u) \,\diff \sigma + \int_\Om b(x,u) (v-u) \,\diff x \\
            &- \sum_{i=1}^k \int_\Om T_i (x,u) (v-u) \,\diff x + \sum_{j=1}^m \int_\Om T^j (x,u) (v-u) \,\diff x \\
            &- \sum_{i=1}^k \int_\Ga U_i (x,u) (v-u) \,\diff \sigma + \sum_{j=1}^m \int_\Ga U^j (x,u) (v-u) \,\diff \sigma \\
            &\ge 0 \quad \text{for all }v\in K.
        \end{split}
    \end{align}
    The inequality above is equivalent to the following variational inequality:  Find $u\in  K$, $\tilde{\eta} = i_{\tau_1(\cdot)}^* \eta  i_{\tau_1(\cdot)}\in [ i_{\tau_1(\cdot)}^* \tilde{f}_0  i_{\tau_1(\cdot)}](u)$, and $\tilde{\zeta} = i_{\tau_2(\cdot)}^* \zeta  i_{\tau_2(\cdot)}\in [ i_{\tau_2(\cdot)}^* \tilde{f}_{0\Ga}  i_{\tau_2(\cdot)}](u)$,
    such that
    \begin{align*}
        \left \langle Au+ \tilde{\eta} +  \tilde{\zeta} +[i_{q(\cdot)}^* \B i_{q(\cdot)}](u) -\sum_{i=1}^k [i_{\tau_1(\cdot)}^* \T_i i_{\tau_1(\cdot)}](u) + \sum_{j=1}^m [i_{\tau_1(\cdot)}^* \T^j  i_{\tau_1(\cdot)}](u)  \right. \\
        \left. -\sum_{i=1}^k [i_{\tau_2(\cdot)}^* \U_i i_{\tau_2(\cdot)}](u) + \sum_{j=1}^m [i_{\tau_2(\cdot)}^* \U^j  i_{\tau_2(\cdot)}](u), v-u \right\rangle \geq 0 \quad\text{for all }v\in K.
    \end{align*}
    This variational inequality is, in its turn, equivalent to finding $u\in D(\pa I_K) =  K$,  $l\in (\pa I_K)(u)$, and
    \begin{align*}
        \tilde{\eta} = i_{\tau_1(\cdot)}^* \eta  i_{\tau_1(\cdot)}\in [ i_{\tau_1(\cdot)}^* \tilde{f}_0  i_{\tau_1(\cdot)}](u), \quad
        \tilde{\zeta} = i_{\tau_2(\cdot)}^* \zeta  i_{\tau_2(\cdot)}\in [ i_{\tau_2(\cdot)}^* \tilde{f}_{0\Ga}  i_{\tau_2(\cdot)}](u)
    \end{align*}
    such that
    \begin{align}
        \label{4-10}
        \begin{split}
            \mathcal{A}(u, l, \tilde{\eta},  \tilde{\zeta})&:=Au + l +\tilde{\eta} + \tilde{\zeta} +[i_\q^* \B i_\q](u)\\
            &\qquad -\sum_{i=1}^k [i_{\tau_1(\cdot)}^* \T_i i_{\tau_1(\cdot)}](u) + \sum_{j=1}^m [i_{\tau_1(\cdot)}^* \T^j  i_{\tau_1(\cdot)}](u) \\
            &\qquad -\sum_{i=1}^k [i_{\tau_2(\cdot)}^* \U_i i_{\tau_2(\cdot)}](u)+ \sum_{j=1}^m [i_{\tau_2(\cdot)}^* \U^j  i_{\tau_2(\cdot)}](u) = 0
        \end{split}
    \end{align}
    in $\WH^*$. We observe that $\partial I_K$ is a maximal monotone mapping and
    \begin{align*}
         & A + i_{\tau_1(\cdot)}^* \tilde{f}_0  i_{\tau_1(\cdot)} +  i_{\tau_2(\cdot)}^* \tilde{f}_{0\Ga}  i_{\tau_2(\cdot)}  + i_\q^* \B i_\q
        - \sum_{i=1}^k i_{\tau_1(\cdot)}^* \T_i i_{\tau_1(\cdot)}\\
        & + \sum_{j=1}^m i_{\tau_1(\cdot)}^* \T^j  i_{\tau_1(\cdot)}
        -\sum_{i=1}^k i_{\tau_2(\cdot)}^* \U_i i_{\tau_2(\cdot)} + \sum_{j=1}^m  i_{\tau_2(\cdot)}^* \U^j  i_{\tau_2(\cdot)}
    \end{align*}
    is a (multi-valued) pseudomonotone bounded mapping from $\WH$ to $2^{\WH^*}$.

    Hence, to apply the abstract existence result in Corollary 2.3 of \cite{le:ret11}, we only need to check the following coercivity condition:  There exists $u_0\in K$ such that
    \begin{equation}
        \label{4-11}
        \lim_{\substack{\| u \| \to \infty\\u\in  K}}
        \l[
            \inf_{\substack{l\in \pa I_K(u)\\ \tilde{\eta} \in [ i_{\tau_1(\cdot)}^* \tilde{f}_0  i_{\tau_1(\cdot)}](u)\\ \tilde{\zeta}\in [ i_{\tau_2(\cdot)}^* \tilde{f}_{0\Ga}  i_{\tau_2(\cdot)}](u)} }
            \l \langle \mathcal{A}(u, l, \tilde{\eta},  \tilde{\zeta}), u-u_0 \r\rangle\r] = \infty,
    \end{equation}
    see \eqref{4-10}.

    In fact, let $u_0$ be any (fixed) element of $K$.  For any $u\in  K$, any $l\in (\pa I_K)(u)$, we have $0 = I_K (u_0) - I_K(u) \ge \la l , u_0 -u\ra$, i.e., $\la l , u- u_0 \ra \ge 0$. Hence, to prove \eqref{4-11}, one only needs to show that
    \begin{equation}
        \label{4-12}
        \inf_{\substack{\tilde{\eta} \in [ i_{\tau_1(\cdot)}^* \tilde{f}_0  i_{\tau_1(\cdot)}](u)\\ \tilde{\zeta}\in [ i_{\tau_2(\cdot)}^* \tilde{f}_{0\Ga}  i_{\tau_2(\cdot)}](u)}}
        \l\langle \hat{\mathcal{A}}(u,  \tilde{\eta},  \tilde{\zeta}), u-u_0 \r\rangle \to \infty
    \end{equation}
    as $\| u \|_{1,\mathcal{H}}\to\infty$, $u\in  K$, where
    \begin{align*}
        \hat{\mathcal{A}}(u,  \tilde{\eta},  \tilde{\zeta})
        & :=Au + \tilde{\eta} + \tilde{\zeta} +[i_\q^* \B i_\q](u) -\sum_{i=1}^k [i_{\tau_1(\cdot)}^* \T_i i_{\tau_1(\cdot)}](u)\\
        & \qquad + \sum_{j=1}^m [i_{\tau_1(\cdot)}^* \T^j  i_{\tau_1(\cdot)}](u)-\sum_{i=1}^k [i_{\tau_2(\cdot)}^* \U_i i_{\tau_2(\cdot)}](u)\\
        & \qquad + \sum_{j=1}^m [i_{\tau_2(\cdot)}^* \U^j  i_{\tau_2(\cdot)}](u).
    \end{align*}

    Let $\tilde{\eta} = i_{\tau_1} ^* \eta i_{\tau_1} \in [i_{\tau_1}^* \tilde{f}_0 i_{\tau_1} ] (u)$ and $\tilde{\zeta} = i_{\tau_2} ^* \zeta i_{\tau_2} \in [i_{\tau_2}^* \tilde{f}_{0\Gamma} i_{\tau_2} ] (u)$, where $\eta \in \tilde{f}_0 (u)$ and $\zeta \in \tilde{f}_{0\Gamma} (u)$.  It follows from \eqref{4-3} that
    \begin{align}\label{4-13}
        \begin{split}
            &|\la \tilde{\eta} , u-u_0\ra|\\
            & \le \l(\| k_\Omega \|_{\tau_1'(\cdot)} + \|\un{\eta} \|_{\tau_1'(\cdot)}  + \|\ov{\eta} \|_{\tau_1'(\cdot)}\r) \l(\| u \|_{\tau_1(\cdot)}  + \| u_0 \|_{\tau_1(\cdot)}\r) \\
            & \le C\l( \| u \|_{\tau_1(\cdot)}  + 1\r)\\
            & \le C\l(\| u \|  + 1\r)
        \end{split}
    \end{align}
    and
    \begin{align}\label{4-13a}
        \begin{split}
            &|\la \tilde{\zeta} , u-u_0\ra|\\
            & \le  \l(\| k_\Gamma \|_{\tau_2'(\cdot),\Gamma} + \|\un{\zeta} \|_{\tau_2'(\cdot),\Gamma}  + \|\ov{\zeta} \|_{\tau_2'(\cdot),\Gamma}\r) \l(\| u \|_{\tau_2(\cdot),\Gamma} + \| u_0 \|_{\tau_2(\cdot),\Gamma}\r) \\
            & \le C\l( \| u \|_{\tau_2(\cdot),\Gamma}  + 1\r)\\
            & \le C\l( \| u \|  + 1\r).
        \end{split}
    \end{align}
    From \eqref{4-5} and \eqref{4-6}, by applying H\"older's and Young's inequalities with $\ep$ for variable exponents (see e.g.\,\cite{le:ssm09}), we get\phantom{\ref{4-13a}}
    \begin{align}\label{4-14}
        \begin{split}
            &\l\langle [i_\q^* \B i_\q] (u) , u-u_0\r\rangle\\
            & \geq  a_2 \int_\Omega |u|^{q(x)} \,\diff x  -\int_\Omega \l(a_1 + C |u|^{q(x) - 1}\r) |u_0| \,\diff x - C \\
            & \geq  \frac{a_2}{2} \int_\Omega |u|^{q(x)} \,\diff x - C.
        \end{split}
    \end{align}
    On the other hand, we have for any $i\in \{1,\dots , k\}$,
    \begin{align*}
        &\l|\l\langle [i_{\tau_1(\cdot)}^* \T_i i_{\tau_1(\cdot)} ](u) , u-u_0\r\rangle \r|\\
         & =  \left|\int_\Omega T_i (x,u) (u-u_0) \,\diff x \right|\\
         &\leq  \l\| \un{\eta}_i -\un{\eta} \r\|_{\tau_1'(\cdot)} \l(\|u \|_{\tau_1(\cdot)} + \|u_0 \|_{\tau_1(\cdot)}\r).
    \end{align*}
    Hence,
    \begin{equation}
        \label{4-15}
        \sum_{i=1}^k \l|\l\langle [i_{\tau_1(\cdot)} ^* \T_i i_{\tau_1(\cdot)} ](u) , u-u_0\r\rangle \r|
        \leq C\l(\|u \|_{\tau_1(\cdot)} +1\r) \leq C(\|u \| +1).
    \end{equation}
    Similarly,\phantom{\ref{4-15}}
    \begin{equation}
        \label{4-16}
        \sum_{j=1}^m \l|\l\langle [i_{\tau_1(\cdot)}^* \T^j  i_{\tau_1(\cdot)} ](u) , u-u_0\r\rangle \r|
        \leq C\l(\|u \|_{\tau_1(\cdot)} +1\r)
        \leq C(\|u \| +1),
    \end{equation}
    and\phantom{\ref{4-16}}
    \begin{align}\label{4-17}
        \begin{split}
            &\sum_{i=1}^k \l|\l\langle [i_{\tau_2(\cdot)} ^* \U_i i_{\tau_2(\cdot)} ](u) , u-u_0\r\rangle \r|, \  \sum_{j=1}^m \l|\l\langle [i_{\tau_2(\cdot)} ^* \U^j  i_{\tau_2(\cdot)} ](u) , u-u_0\r\rangle \r|\\
            &\leq C\l(\|u \|_{\tau_2(\cdot),\Gamma} +1\r) \leq C(\|u \| +1).
        \end{split}
    \end{align}
    Lastly, since $A$ has as a potential functional the following convex functional \phantom{\ref{4-17}}
    \begin{align*}
        I(u)=\int_{\Omega}\left[\frac{|\nabla u|^{p(x)}}{p(x)}+\mu(x) \frac{|\nabla u|^{q(x)}}{q(x)}\right] \,\diff x,
    \end{align*}
    we see that
    \begin{equation}
        \label{4-18}
        \langle A u , u- u_0 \rangle \geq I(u) - I(u_0) = I(u) - C.
    \end{equation}
    On the other hand, it follows from (H0) that\phantom{\ref{4-18}}
    \begin{align*}
        \int_\Omega |u|^{q(x)} \,\diff x \geq \int_\Omega |u|^{p(x)} \,\diff x - |\Omega|,
    \end{align*}
    where $|\Omega|$ is the Lebesgue measure of $\Omega$.  Hence, it follows from \eqref{4-14} that there is $a_3 >0$ such that for all $u\in \WH$
    \begin{align}
        \label{4-18a}
        \begin{split}
            \l\langle i_{\q}^* \B i_\q (u) , u - u_0\r\rangle
            & \geq a_3 \int_\Omega \l[|u|^{p(x)} + \mu(x)|u|^{q(x)}\r]\,\diff x - C \\
            & \ge a_3 \int_\Omega \left[\frac{|u|^{p(x)}}{p(x)} + \mu(x)\frac{|u|^{q(x)}}{q(x)}\right]\,\diff x - C.
        \end{split}
    \end{align}

    Combining the estimates from \eqref{4-13} to \eqref{4-18a}, we see that for any $u\in  K$, $\tilde{\eta}\in [i_{\tau_1(\cdot)} ^* \tilde{f}_0 i_{\tau_1(\cdot)}](u)$ and $\tilde{\zeta}\in [i_{\tau_2(\cdot)} ^* \tilde{f}_{0\Gamma} i_{\tau_2(\cdot)}](u)$
    \begin{align}
        \label{4-18b}
        \begin{split}
            &\bigg\la Au+ \tilde{\eta} +\tilde{\zeta} +  [i_\q^* \B i_\q](u)  -\sum_{i=1}^k [i_{\tau_1(\cdot)}^* \T_i i_{\tau_1(\cdot)}](u) \\
            &\quad + \sum_{j=1}^m [i_{\tau_1(\cdot)}^* \T^j  i_{\tau_1(\cdot)}](u) \ds -\sum_{i=1}^k [i_{\tau_2(\cdot)}^* \U_i i_{\tau_2(\cdot)}](u)\\
            &\quad +\sum_{j=1}^m [i_{\tau_2(\cdot)}^* \U^j  i_{\tau_2(\cdot)}](u) , u-u_0 \bigg\ra \\
            &\ge  \ds \min\{1, a_3\}\int_\Om \left[\frac{|\nabla u|^{p(x)}}{p(x)}+\mu(x) \frac{|\nabla u|^{q(x)}}{q(x)}\r.\\
            &\l.\qquad\qquad \qquad\qquad+ \frac{|u|^{p(x)}}{p(x)} + \mu(x)\frac{|u|^{q(x)}}{q(x)}\right] \,\diff x
            - C (\|u \|+1).
        \end{split}
    \end{align}
    Since
    \begin{align*}
        \lim_{\|u \|_{1,\mathcal{H}}\to \infty}\frac{1}{\|u \|_{1,\mathcal{H}}}&\int_\Om \left[\frac{|\nabla u|^{p(x)}}{p(x)}+\mu(x) \frac{|\nabla u|^{q(x)}}{q(x)} \r.\\
        &\l. \qquad\qquad+ \frac{|u|^{p(x)}}{p(x)} + \mu(x)\frac{|u|^{q(x)}}{q(x)}\right] \,\diff x = \infty,
    \end{align*}
    see Proposition 3.5 in \cite{CGHW}, the estimate in \eqref{4-18b} implies \eqref{4-12}. It follows from Corollary 2.3 in \cite{le:ret11} that there exist $u,\eta$, and $\zeta$ that satisfy \eqref{4-7} and \eqref{4-8}.

    In the next step, we show that
    \begin{equation}
        \label{4-19}
        \un{u}_s \le u \le \ov{u}_r \quad\text{a.\,e.\,in } \Omega,
    \end{equation}
    for all $s\in \{1,\dots , k\}$ and for all $r\in \{1,\dots , m\}$.  In fact, let $s\in \{1,\dots , k\}$.  By putting $v= \un{u}_s \vee u = u +  (\un{u}_s -u)^+ \in K$ into \eqref{4-8}, we obtain
    \begin{align}\label{4-20}
        \begin{split}
            &\l\la Au , (\un{u}_s -u)^+ \r\ra +\int_\Om \eta (\un{u}_s -u)^+  \,\diff x + \int_\Gamma \zeta (\un{u}_s -u)^+  \,\diff \sigma\\
            & + \int_\Om b(x,u) (\un{u}_s -u)^+ \,\diff x - \sum_{i=1}^k \int_\Om T_i (x,u) (\un{u}_s -u)^+  \,\diff x\\
            & + \sum_{j=1}^m \int_\Om T^j (x,u) (\un{u}_s -u)^+  \,\diff x - \sum_{i=1}^k \int_\Gamma U_i (x,u) (\un{u}_s -u)^+  \,\diff \sigma\\
            &+ \sum_{j=1}^m \int_\Gamma U^j (x,u) (\un{u}_s -u)^+  \,\diff \sigma \ge 0.
        \end{split}
    \end{align}
    Since $\un{u}_s$ is a subsolution of \eqref{problem}, we have from its definition that there are $\un{\eta}_s\in L^{\tau_1'(\cdot)} (\Om)$ and $\un{\zeta}_s\in L^{\tau_2'(\cdot)} (\Gamma)$ satisfying conditions (i)-(iii) in Definition \ref{D102} with $\un{u}$, $\un{\eta}$, $\un{\zeta}$ replaced by $\un{u}_s$, $\un{\eta}_s$, $\un{\zeta}_s$.

    Letting $v = \un{u}_s - (\un{u}_s - u)^+ = \un{u}_s \wedge u \in \un{u}_s \wedge K$ in Definition \ref{D102} (iii) (with $\un{u}_s$, $\un{\eta}_s$, and $\un{\zeta}_s$) yields
    \begin{equation}
        \label{4-21}
        - \l\la A \un{u}_s ,  (\un{u}_s - u)^+\r\ra -\int_\Om \un{\eta}_s  (\un{u}_s - u)^+ \,\diff x  -\int_\Gamma \un{\zeta}_s  (\un{u}_s - u)^+ \,\diff \sigma  \ge 0 .
    \end{equation}
    Adding \eqref{4-20} and \eqref{4-21} gives us
    \begin{align*}
         & \l\langle Au - A \un{u}_s ,  (\un{u}_s - u)^+\r\rangle + \int_\Om (\eta - \un{\eta}_s)  (\un{u}_s - u)^+ \,\diff x\\
         &+ \int_\Gamma (\zeta - \un{\zeta}_s)  (\un{u}_s - u)^+ \,\diff \sigma + \int_\Om b(x,u) (\un{u}_s -u)^+ \,\diff x\\
         &- \sum_{i=1}^k \int_\Om T_i (x,u) (\un{u}_s -u)^+  \,\diff x
         + \sum_{j=1}^m \int_\Om T^j (x,u) (\un{u}_s -u)^+  \,\diff x\\
         &- \sum_{i=1}^k \int_\Gamma U_i (x,u) (\un{u}_s -u)^+  \,\diff \sigma+ \sum_{j=1}^m \int_\Gamma U^j (x,u) (\un{u}_s -u)^+  \,\diff \sigma \ge 0.
    \end{align*}
    First, note that
    \begin{align*}
        &\l\langle Au - A \un{u}_s  ,  (\un{u}_s - u)^+\r\rangle\\
         & = \int_{\{x\in \Omega \,:\, \un{u}_s (x) \ge u(x)\}} \l[\l(|\nabla {u}|^{p(x)-2} \nabla {u}+ \mu(x) |\nabla {u}|^{q(x)-2} \nabla {u}\r)\r.                                                       \\
         & \qquad\l. -\l(|\nabla \underline{u}_s|^{p(x)-2} \nabla \underline{u}_s+ \mu(x) |\nabla \underline{u}_s|^{q(x)-2} \nabla \underline{u}_s\r)\r] \cdot \nabla (\underline{u}_s -u)\, \diff x\le 0.
    \end{align*}
    At $x\in \Om$ such that $\un{u}_s > u(x)$, since $\un{u}_s (x) \le \un{u}(x) \le \ov{u}(x)$, we have
    \begin{align*}
        \int_\Om T^j (x,u)(\un{u}_s -u)^+ \,\diff x
        = \int_{\{x\in\Omega \,:\, \un{u}_s(x) > u(x)\}} T^j (x,u)(\un{u}_s -u) \,\diff x = 0,
    \end{align*}
    for all $j\in \{1,\dots , m\}$.  Furthermore,  $\eta(x) \in \{\un{\eta}(x)\}$, i.e., $\eta(x) = \un{\eta}(x) $. Also, for such $x$, we have $ T_s(x,u(x)) = | \un{\eta}_s(x) - \un{\eta}(x) | $ and
    \begin{align*}
        \int_\Omega T_i(x,u) (\un{u}_s -u)^+ \,\diff x \ge 0\quad \text{for all }i\in \{1,\dots , k\}.
    \end{align*}
    Therefore,
    \begin{align*}
         & \ds \int_\Om (\eta -\un{\eta}_s)(\un{u}_s - u)^+ \,\diff x - \sum_{i=1}^k \int_\Om T_i(x,u)(\un{u}_s - u)^+ \,\diff x                                              \\
         & \le  \ds \int_\Om (\eta -\un{\eta}_s)(\un{u}_s - u)^+ \,\diff x - \int_\Om T_s (x,u)(\un{u}_s - u)^+ \,\diff x                                                     \\
         & =  \ds \int_{\{x\in\Omega \,:\, \un{u}_s(x) > u(x)\}} \l( (\un{\eta} (x) - \un{\eta}_s (x) ) - |\un{\eta} (x) - \un{\eta}_s (x)|\r)[\un{u}_s (x) - u(x)] \,\diff x \\
         & \le  0.
    \end{align*}
    Similarly, we have
    \begin{align*}
        \int_\Gamma U^j (x,u)(\un{u}_s -u)^+ \,\diff \sigma & =  0  \quad \text{for all }j\in \{1,\dots , m\},   \\
        \int_\Gamma U_i(x,u) (\un{u}_s -u)^+ \,\diff \sigma & \ge 0 \quad \text{for all }i\in \{1,\dots , k\} ,
    \end{align*}
    and
    \begin{align*}
         & \ds \int_\Gamma (\zeta -\un{\zeta}_s)(\un{u}_s - u)^+ \,\diff \sigma - \sum_{i=1}^k \int_\Gamma U_i(x,u)(\un{u}_s - u)^+ \,\diff \sigma                                    \\
         & \le  \ds \int_\Gamma (\zeta -\un{\zeta}_s)(\un{u}_s - u)^+ \,\diff \sigma - \int_\Gamma U_s (x,u)(\un{u}_s - u)^+ \,\diff \sigma                                           \\
         & =  \ds \int_{\{x\in\Gamma\,:\, \un{u}_s(x) > u(x)\}} \l( (\un{\zeta} (x) - \un{\zeta}_s (x) ) - |\un{\zeta} (x) - \un{\zeta}_s (x)|\r)[\un{u}_s (x) - u(x)] \,\diff \sigma \\
         & \le  0.
    \end{align*}

    Combining the above inequalities, we obtain
    \begin{align*}
        0 \le \int_\Om b(x,u) (\un{u}_s - u)^+ \,\diff x  = \int_{\{x\in\Om \,:\, \un{u}_s(x) > u(x)\}}  b(x,u) (\un{u}_s - u) \,\diff x.
    \end{align*}
    From \eqref{4-4}, if $\un{u}_s(x) > u(x)$ then $\un{u} > u(x)$ and $b(x,u(x)) = - [\un{u}(x) - u(x)]^{q(x)-1}$.  Hence,
    \begin{align*}
        0 \le  -  \int_{\{x\in\Omega \,:\, \un{u}_s(x) > u(x)\}}  (\un{u}(x) - u(x))^{q(x)-1}  [\un{u}_s (x) - u(x)] \,\diff x.
    \end{align*}
    Since $\un{u}(x) - u(x) > 0$ and $\un{u}_s(x) - u(x) > 0$ on the set $\{x\in\Om : \un{u}_s(x) > u(x)\}$, this inequality implies that this set has measure $0$, which means that $u(x) \geq \un{u}_s (x)$ for a.\,a.\,$x\in \Omega$. The second inequality in \eqref{4-19} is demonstrated in the same way.

    As a consequence of \eqref{4-19}, we see that $\un{u} \leq u \leq \ov{u}$ a.\,e.\,in $\Omega$ and thus their traces on $\Gamma$ also satisfy $\un{u} \leq u \leq \ov{u}$ a.\,e.\,on $\Gamma$. This implies that $b(\cdot , u) = T_i (\cdot , u) = T^j(\cdot , u) = 0$ a.\,e.\,in $\Omega$, $U_i (\cdot , u) = U^j(\cdot , u) = 0$ a.\,e.\,on $\Gamma$ for all $i\in \{ 1, \dots , k\}$, for all $j\in \{ 1, \dots , m\}$ and also $f_0(x,u(x)) = f(x,u(x))$ for a.\,a.\,$x\in \Om$ and $f_{0\Gamma}(x,u(x)) = f_\Gamma(x,u(x))$ for a.\,a.\,$x\in \Gamma$.  This shows that $u$ is a solution of \eqref{problem} which completes the proof of Theorem \ref{T4-1}.
\end{proof}

The proof of  Theorem \ref{T102} is now an immediate  consequence of Theorem \ref{T4-1}.

\begin{proof}[Proof of Theorem \ref{T102}]
    In the particular case where $m=n=1$, condition \eqref{g1} becomes condition (F3) and Theorem \ref{T4-1} reduces to Theorem \ref{T102}.
\end{proof}

\section{Extremal Solutions: Proof of Theorem \ref{T103}}\label{section_5}

In this section we give the proof of Theorem \ref{T103}.

\begin{proof}[Proof of Theorem \ref{T103}]
    (i) Since $\un{u}, \ov{u}\in \WH$, it follows  that the set $\{ \| u \|_{\mathcal{H}} \,:\, u\in \S\}$ is bounded.  Let $\{ u_n\}_{n\in\N}$ be a sequence in $\S$ and $\{ \eta_n\}_{n\in\N} \subset L^{\tau_1'(\cdot)}(\Om)$, $\{ \zeta_n\}_{n\in\N}\subset L^{\tau_2'(\cdot)}(\Gamma)$ be  corresponding sequences that satisfy \eqref{problem1} (for each $u=u_n$ and $\eta = \eta_n, \zeta = \zeta_n$).

    From (F3), $\{ \eta_n\}_{n\in\N}$ is a bounded sequence in $L^{\tau_1'(\cdot)}(\Om)$ and $\{ \zeta_n\}_{n\in\N}$ is a bounded sequence in $L^{\tau_2'(\cdot)}(\Gamma)$.  Using \eqref{problem1} with $u_n$, $\eta_n$, $\zeta_n$, and $v=v_0$, a fixed element of $K$, we see that $\{ \rho_\Hi(|\nabla u_n|)\}_{n\in\N}$ is a bounded sequence and thus the set $\{ \| \nabla  u_n \|_{\mathcal{H}} \,:\, n\in \N\}$ is also bounded.  Hence, $\{ u_n\}_{n\in\N}$ is a bounded sequence in $\WH$ and there exists a subsequence $\{ u_{n_l}\}_{l\in\N}\subset \{ u_n\}_{n\in\N}$ such that $u_{n_l} \rh u_0$ in $\WH$ for some $u_0\in K$ (note that $K$ is weakly closed in $\WH$).  Thus, $u_{n_l} \to u_0$ in $L^\Hi(\Om)$ and in $L^{\tau_1(\cdot)}(\Om)$, and $u_{n_l}|_{\Gamma} \to u_0|_{\Gamma}$ in  $L^{\tau_2(\cdot)}(\Gamma)$.

    By passing to  a subsequence if necessary, we can also assume that $u_{n_l} \to u_0$ a.\,e.\,in $\Om$ and $u_{n_l}|_\Gamma \to u_0|_\Gamma$ a.\,e.\,on $\Gamma$. Because of the boundedness of $\{\eta_n\}_{n\in\N}$ in $L^{\tau_1'(\cdot)}(\Om)$ and of $\{\zeta_n\}_{n\in\N}$ in $L^{\tau_2'(\cdot)}(\Gamma)$, $\eta_{n_l} \rh \eta_0$ in $L^{\tau_1'(\cdot)}(\Om)$ for some $\eta_0\in L^{\tau_1'(\cdot)}(\Om)$ and $\zeta_{n_l} \rh \zeta_0$ in $L^{\tau_2'(\cdot)}(\Gamma)$ for some $\zeta_0\in L^{\tau_2'(\cdot)}(\Gamma)$.
    Due to the compactness of $i_{\tau_1(\cdot)}$ and $i_{\tau_2(\cdot)}$, and thus of $i_{\tau_1(\cdot)}^*$ and $i_{\tau_2(\cdot)}^*$, we have $i_{\tau_1(\cdot)}^* \eta_{n_l} \to i_{\tau_1(\cdot)}^* \eta_{0}$  and $i_{\tau_2(\cdot)}^* \zeta_{n_l} \to i_{\tau_2(\cdot)}^* \zeta_{0}$  in $\WH^*$. Therefore
    \begin{equation}
        \label{5-1}
        \int_\Om \eta_{n_l} (u_{n_l}-u_0) \,\diff x \to 0
        \quad\text{and}\quad
        \int_\Gamma \zeta_{n_l} (u_{n_l}-u_0) \,\diff \sigma \to 0 \quad \text{as } l\to\infty.
    \end{equation}
    From \eqref{problem1} with $u = u_{n_l}$ and $v= u_0$, we see that
    \begin{align*}
        \liminf_{l\to\infty}\, \int_\Om \l(|\nabla u_{n_l}|^{p(x)-2} \nabla u_{n_l}+ \mu(x) |\nabla u_{n_l}|^{q(x)-2} \nabla u_{n_l}\r)\cdot \nabla (u_{n_l}-u_0)\,\diff x \le 0.
    \end{align*}
    We obtain $u_{n_l}\to u_0$  in $\WH$ due to the $(\Ss_+)$-property of the operator $A\colon$ $\WH$ $\to \WH^*$, see Proposition \ref{prop1}.

    Next, let us prove that $u_0\in \S$.  It is evident that
    \begin{equation}
        \label{5-2}
        \un{u}\leq u_0 \leq \ov{u} \quad \text{a.\,e.\,in } \Omega.
    \end{equation}
    Let $f_0$ and $f_{0\Gamma}$ be defined by \eqref{truncated1} and \eqref{truncated2} in the proof of Theorem \ref{T4-1}.  Since $\un{u} \leq u_n \leq \ov{u}$ a.\,e.\,in $\Omega$, we see that $u_n$ and $\eta_n$, $\zeta_n$ satisfy \eqref{problem1} with $f_0$ and $f_{0\Gamma}$ instead of $f$ and $f_{\Gamma}$.  From \eqref{5-1} and the fact that $i_{\tau_1(\cdot)}^* \tilde{f}_0 i_{\tau_1(\cdot)}$ and  $i_{\tau_2(\cdot)}^* \tilde{f}_{0\Gamma} i_{\tau_2(\cdot)}$ are generalized pseudomonotone from $\WH$ to $\WH^*$ (cf.\,Proposition \ref{p35}), we have $\eta_0 \in [i_{\tau_1(\cdot)}^* \tilde{f}_0 i_{\tau_1(\cdot)}](u_0)$ and $\zeta_0 \in [i_{\tau_2(\cdot)}^* \tilde{f}_{0\Gamma} i_{\tau_2(\cdot)}](u_0)$, that is,
    \begin{align*}
        \eta_0(x)  & \in f_0(x,u_0(x)) = f(x,u_0(x)) \quad \text{for a.\,a.\,} x\in \Omega,                 \\
        \zeta_0(x) & \in f_{0\Gamma}(x,u_0(x)) = f_\Gamma(x,u_0(x)) \quad \text{for a.\,a.\,} x\in \Gamma,
    \end{align*}
    and
    \begin{align*}
        \int_\Om \eta_{n_l} u_{n_l} \,\diff x \to \int_\Om \eta_0 u_0 \,\diff x,
        \quad
        \int_\Gamma \zeta_{n_l} u_{n_l} \,\diff \sigma \to \int_\Gamma \zeta_0 u_0 \,\diff \sigma.
    \end{align*}
    Therefore, for all $v\in K$,
    \begin{align*}
         & \ds \int_\Om \l(|\nabla u_{n_l}|^{p(x)-2} \nabla u_{n_l}+ \mu(x) |\nabla u_{n_l}|^{q(x)-2} \nabla u_{n_l}\r) \cdot \nabla (v-u_{n_l})\,\diff x\\
         &+ \int_\Om \eta_{n_l} (v-u_{n_l})\, \diff x + \int_\Gamma \zeta_{n_l} (v-u_{n_l})\, \diff \sigma                                                                                                                                                                                  \\
         & \to \int_\Omega \l( |\nabla u_{0}|^{p(x)-2} \nabla u_{0}+ \mu(x) |\nabla u_{0}|^{q(x)-2} \nabla u_{0}\r) \cdot \nabla (v-u_{0})\,\diff x\\
         & + \int_\Om \eta_{0} (v-u_{0})\,\diff x + \int_\Gamma \zeta_{0} (v-u_{0})\, \diff \sigma.
    \end{align*}
    Since $u_{n_l} \in \S$, this limit  shows that $u_0$ and $\eta_0$, $\zeta_0$ satisfy \eqref{problem1} which in view of \eqref{5-2} implies that $u_0\in \S$.  We thus obtain the compactness of $\S$ in $\WH$.

    (ii) The proofs for (ii) and (iii) follow the same lines as those for the case of regular Sobolev spaces, thus their outlines are presented here for the sake of completeness.  Assuming (\ref{lattice}), we see that  if $u_0\in \S$ then $u_0\wedge K\subset K$ and thus $u_0$ is  a subsolution of \eqref{problem} in the sense of Definition \ref{D102}.  If $u_1, u_2\in \S$ then they are subsolutions of \eqref{problem} and Theorem \ref{T4-1} thus implies the existence of a solution $u$ of \eqref{problem} such that $\max\{ u_1,u_2\}\le u\le \min \{ \ov{u}_j \,:\, 1\le j \le m\} = \ov{u}$.  It is clear that $u\in \S$.  In fact, since $u_1, u_2 \in \S$, this follows directly from the definition of $\S$ and  the inequalities
    \begin{align*}
    	\underline{u} \leq u_1 \leq \max \{ u_1, u_2 \} \leq u \leq \overline{u}.
    \end{align*}

    (iii)  Since $\WH$ is separable (with the norm topology), so is $\S$.  Let $\{w_n\}_{n\in\N}$ be a dense sequence in $\S$.  Using the directedness of $\S$, we can construct inductively a sequence $\{ u_n\}_{n\in\N}$ in $\S$ such that $w_n \leq u_n \leq u_{n+1}$ for all $n\in \N$.  Let
    \begin{align*}
        u^*(x) = \sup\{ u_n(x) \,:\, n\in \N\} = \lim_{n\to\infty} u_n(x) \quad \text{for }x\in \Om.
    \end{align*}
    As a consequence of the compactness of $\S$, $u_n \to u^*$ in $\WH$ and $u^* \in \S$.  Since $u^* \geq w_n$ a.\,e.\,in $\Om$ for all $n\in \N$, from the density of $\{w_n\}_{n\in\N}$ in $\S$, we see that $u^* \geq u$ a.\,e.\,in $\Omega$ for all $u\in \S$.  The existence of the smallest element $u_*$ of $\S$ is proved analogously.
\end{proof}

\section{Application: Generalized Variational-Hemivariational Inequalities}\label{section_6}

In this section we are dealing with the generalized variational-hemi\-variat\-ional inequality \eqref{106} which is of the form
\begin{align}\label{601}
    \begin{split}
        u\in K\,:\, &\langle  Au, v-u\rangle+\int_{\Omega} j^\circ(\cdot,u,u; v-u)\,\diff x\\ &+\int_{\Gamma} j_{\Gamma}^\circ(\cdot, u, u;  v- u)\,\diff\sigma\ge 0\quad \text{for all }v\in K,
    \end{split}
\end{align}
where $A$ is the variable exponent double-phase operator given by \eqref{103}. The functions $j$, $j_{\Gamma}$ given by
\begin{align*}
    j          & \colon \Omega\times\R\times\R\to \R \quad \text{with}\quad (x,r,s)\mapsto j(x,r,s),            \\
    j_{\Gamma} & \colon \Gamma\times\R\times\R\to \R  \quad \text{with}\quad (x,r,s)\mapsto j_{\Gamma}(x,r,s),
\end{align*}
are supposed to be locally Lipschitz with respect to $s$, and $j^\circ(x,r,s;\varrho)$ and  $j_{\Gamma}^\circ(x,r,s;\varrho)$ denote Clarke's generalized directional derivatives at $s$ in the direction $\varrho$ for fixed $(x,r)$. In case $j$ and  $j_{\Gamma}$ are independent of $r$, (\ref{601}) represents a variational-hemivariational inequality. However, in the general case of problem (\ref{601}) the functions $s\mapsto  j(x,s,s)$ and $s\mapsto  j_{\Gamma}(x,s,s)$ may be not locally Lipschitz but only partially locally Lipschitz. This enlarges the class of variational-hemivariational inequalities considerably, and therefore we are calling them generalized variational-hemivariational  inequalities. Under hypotheses specified next we are going to show that problem (\ref{601}) is equivalent to some subclass of multi-valued variational inequalities of the form \eqref{problem}, which in a sense fills a gap in the literature where both problems are considered independently and separately.

We suppose the following hypotheses on $j$ and $j_{\Gamma}$:
\begin{enumerate}
    \item[\textnormal{(J1)}]
        The functions $x\mapsto j(x,r,s)$ and $x\mapsto j_{\Gamma}(x,r,s)$ are measurable in $\Omega$ and on $\Gamma$, respectively, for all $r,s\in \mathbb{R}$. The functions $r\mapsto j(x,r,s)$ and $r\mapsto j_{\Gamma}(x,r,s)$ are continuous for a.\,a.\,$x\in \Omega$ and $x\in \Gamma$, respectively, and for all $s\in \mathbb{R}$. The functions $s\mapsto j(x,r,s)$ and $s\mapsto j_{\Gamma}(x,r,s)$ are locally Lipschitz for a.\,a.\,$x\in \Omega$ and $x\in \Gamma$, respectively, and for all $r\in \mathbb{R}$.
    \item[\textnormal{(J2)}]
        Let $s\mapsto \partial j(x,r,s)$ and   $s\mapsto \partial j_{\Gamma}(x,r,s)$ denote Clarke's generalized gradient of the functions $j$ and $j_{\Gamma}$ with respect to the variable $s$, respectively. Assume the following growth conditions for $s\to \partial j(x,s,s)$ and $s\to  \partial j_{\Gamma}(x,s,s)$:
        \begin{enumerate}
            \item[]
                There exist $r_1\in C(\close)$, $r_2 \in C(\Gamma)$ with $1<r_1(x)<p^*(x)$ for all $x\in \close$, $1<r_2(x)<p_*(x)$ for all $x\in\Gamma$, $\beta\geq 0$, $\beta_\Gamma\geq 0$ and functions $\alpha\in\Lp{r_1'(\cdot)}$, $\alpha_\Gamma \in L^{r_2'(\cdot)}(\Gamma)$  such that
                \begin{align*}
                    \sup \l\{|\eta| \,:\, \eta \in \partial j(x,s,s)\r\} \leq \alpha(x)+\beta |s|^{r_1(x)-1}
                \end{align*}
                for a.\,a.\,$x\in\Omega$, for all $s\in\R$, and
                \begin{align*}
                    \sup \l\{|\zeta| \,:\, \zeta \in \partial j_{\Gamma}(x,s,s)\r\} \leq \alpha_\Gamma(x)+\beta_\Gamma |s|^{r_2(x)-1}
                \end{align*}
                for a.\,a.\,$x\in\Gamma$, and for all $s\in\R$.
        \end{enumerate}
    \item[\textnormal{(J3)}]
        Let $s\mapsto j^\circ(x,r,s;\varrho)$ and   $s\mapsto j^\circ_{\Gamma}(x,r,s;\varrho)$ denote Clarke's generalized directional derivative of the functions $s\mapsto  j(x,r,s)$ and   $s\mapsto  j_{\Gamma}(x,r,s)$ at $s$, respectively, in the direction $\varrho$ for fixed $(x,r)$. Suppose that $s\mapsto  j^\circ(x,s,s;\varrho)$ and   $s\mapsto j^\circ_{\Gamma}(x,s,s;\varrho)$ are upper semicontinuous for a.\,a.\,$x\in \Omega$, and $x\in \Gamma$, respectively, and for all $\varrho\in \mathbb{R}$.
\end{enumerate}

Let us define the multi-valued functions $f\colon \Omega\times\R\to 2^{\R}$ and $f_{\Gamma}\colon \Gamma\times\R\to 2^{\R}$ as follows:
\begin{equation}
    \label{604}
    f(x,s)=\partial j(x,s,s),\quad f_{\Gamma}(x,s)=\partial j_{\Gamma}(x,s,s).
\end{equation}

For the so defined multi-valued functions the following lemma holds true.

\begin{proposition}
    \label{L601}
    Under the assumptions \textnormal{(J1)--(J3)}, the multi-valued functions $f\colon \Omega\times\R\to 2^{\R}$ and $f_{\Gamma}\colon \Gamma\times\R\to 2^{\R}$ defined by \eqref{604} satisfy hypotheses \textnormal{(F1)--(F2)}.
\end{proposition}

\begin{proof}
    Hypothesis (F2) follows immediately from (J2). The proof of property (F1) is just a slight adaption of the proof of \cite[Lemma 3.2]{carl:mvi21}, and therefore can be omitted.
\end{proof}

With the multi-valued functions  $f$ and $f_{\Gamma}$ given by \eqref{604}, respectively, we consider the following associated multi-valued variational inequality: Find $u\in K\subset V_{\Gamma_0}\subset \WH$,  such that there exist $\tau_1\in C(\close)$, $\tau_2\in C(\Gamma)$, $1<\tau_1(x)<p^*(x)$ for all $x\in\close$, $1<\tau_2(x)<p_*(x)$ for all $x\in\Gamma$ and $\eta\in\Lp{\tau'_1(\cdot)}$, $\zeta\in L^{\tau'_2(\cdot)}(\Gamma)$ satisfying  $\eta(x) \in f(x,u(x))$ for a.\,a.\,$x\in\Omega$, $\zeta(x) \in f_\Gamma(x,u(x))$ for a.\,a.\,$x\in\Gamma$ and
\begin{align}\label{605}
    \begin{split}
        & \into \l(|\nabla u|^{p(x)-2} \nabla u+ \mu(x) |\nabla u|^{q(x)-2} \nabla u\r) \cdot \nabla (v-u)\,\diff x\\
        &+\into \eta(v-u) \,\diff x+\int_\Gamma \zeta(v-u)\,\diff \sigma \geq 0
    \end{split}
\end{align}
for all $v\in K$.

We are going to show that the   generalized variational-hemivariational inequality \eqref{601} and the related multi-valued variational inequality \eqref{605} are indeed equivalent provided $K$ satisfies the following lattice property
\begin{equation}
    \label{606}
    K\wedge K\subset K
    \quad\text{and}\quad
    K\vee K\subset K.
\end{equation}

The following equivalence result holds:

\begin{theorem}
    \label{T601}
    Assume hypotheses \textnormal{(H0)} and \textnormal{(J1)--(J3)}, and let the lattice condition \eqref{606} for $K$ be satisfied. Then $u$ is a solution of the generalized variational-hemivariational inequality \eqref{601} if and only if $u$ is a solution of the multi-valued variational inequality \eqref{605} with multi-functions $f$ and $f_{\Gamma}$ given by \eqref{604}.
\end{theorem}

\begin{proof}
    Let $u$ be a solution of (\ref{605}), which due to (\ref{604}) means there exist $\eta\in\Lp{\tau'_1(\cdot)}$ and  $\zeta\in L^{\tau'_2(\cdot)}(\Gamma)$ satisfying
    \begin{align*}
        \eta(x)  & \in f(x,u(x))=\partial j(x,u(x),u(x))\quad \text{for a.\,a.\,} x\in\Omega,              \\
        \zeta(x) & \in f_\Gamma(x,u(x))=\partial j_\Gamma(x,u(x),u(x))\quad \text{for a.\,a.\,}x\in\Gamma
    \end{align*}
    and \eqref{605}. By the definition of $\partial j$ and $\partial j_{\Gamma}$ we get for any $v\in K$
    \begin{equation}
        \label{607}
        \begin{aligned}
             & j^\circ(x,u,u;v-u)\ge \eta(x)\, (v-u) \quad \text{a.\,e.\,in } \Omega,                   \\
             & j_{\Gamma}^\circ(x, u, u; v- u)\ge \zeta(x)\, ( v- u) \quad  \text{a.\,e.\,on } \Gamma.
        \end{aligned}
    \end{equation}

    By (J1) and (J2) we can ensure that the left-hand sides of \eqref{607} belong to $L^1(\Omega)$ and $L^1(\Gamma)$, respectively, which in view of \eqref{605} implies \eqref{601}.

    To prove the reverse, let $u$ be a solution of \eqref{601}. In order to show that $u$ is a solution of the multi-valued variational inequality \eqref{605}, we are going to show that $u$ is both a subsolution and a supersolution of the multi-valued variational inequality \eqref{605}, which then by Proposition \ref{L601} and applying Theorem \ref{T102} yields the existence of a solution $\tilde u$ of \eqref{605} satisfying $\tilde u\in [u,u]$ and thus $u=\tilde u$ completing the proof.

    Let us show first that the solution $u$ of \eqref{601} is a subsolution of the multi-valued variational inequality \eqref{605}. Since $K$ has the lattice property \eqref{606}, we can use in \eqref{601}, in particular, $v\in u\wedge K$, i.e., $v=u\wedge \varphi=u-(u-\varphi)^+$ with $\varphi\in K$, which yields
    \begin{align*}
        &\l\langle Au, -(u-\varphi)^+\r\rangle+\int_{\Omega} j^\circ\l(x,u,u;-(u-\varphi)^+\r)\,\diff x\\
        &+\int_{\Gamma} j_{\Gamma}^\circ\l(x, u, u; -(  u- \varphi)^+\r)\,\diff\sigma \ge 0 \quad\text{for all }\varphi\in K.
    \end{align*}
    From Clarke's calculus we have that  $\varrho \mapsto j^\circ(\cdot,r,s; \varrho)$ (resp. $\varrho \mapsto j_{\Gamma}^\circ(\cdot,r,s; \varrho)$) is positively homogeneous (see Proposition \ref{P1} (i)), so the last inequality is  equivalent  to
    \begin{align*}
        &\l\langle Au, -(u-\varphi)^+\r\rangle+\int_{\Omega} j^\circ(x,u,u;-1)(u-\varphi)^+\,\diff x\\
        &+\int_{\Gamma} j_{\Gamma}^\circ(x,  u,  u; -1)(  u-  \varphi)^+\,\diff\sigma\ge 0\quad\text{for all }\varphi\in K.
    \end{align*}
    Using again for any $v\in u\wedge K$ its representation in the form $v=u-(u-\varphi)^+$  with $\varphi\in K,$ the last inequality yields
    \begin{align}\label{608}
        \begin{split}
            &\langle Au, v-u\rangle+\int_{\Omega} -j^\circ(x,u,u;-1)(v-u)\,\diff x\\
            &+\int_{\Gamma} -j_{\Gamma}^\circ(x,  u,  u; -1)(  v-  u) \,\diff\sigma\ge 0\quad\text{for all } v\in u\wedge K.
        \end{split}
    \end{align}
    From Clarke's calculus (see Proposition \ref{P1} (iv)) we get
    \begin{equation}\label{609}
        \begin{aligned}
            &j^\circ(x,u(x),u(x);-1)\\
             & =\max\{-\theta(x)\,:\, \theta(x)\in \partial j(x,u(x),u(x))\}                         \\
             & =-\min\{\theta(x)\,:\, \theta(x)\in \partial j(x,u(x),u(x))\}=:-\underline{\eta}(x),
        \end{aligned}
    \end{equation}
    where
    \begin{equation}\label{610}
        \underline{\eta}(x) \in \partial j(x,u(x),u(x)) \quad \text{for all }x\in\Omega.
    \end{equation}
    Similarly, we get for $j_{\Gamma}^\circ$
    \begin{equation} \label{611}
        \begin{aligned}
            &j_{\Gamma}^\circ(x, u(x), u(x);-1)\\
             & =\max\{-\zeta(x)\,:\, \zeta(x)\in \partial j_{\Gamma}(x,u(x),u(x))\}                           \\
             & =-\min\{\zeta(x)\,:\, \zeta(x)\in \partial j_{\Gamma}(x, u(x),u(x))\}=:-\underline{\zeta}(x),
        \end{aligned}
    \end{equation}
    with \phantom{\ref{611}}
    \begin{equation}
        \label{612}
        \underline{\zeta}(x) \in \partial j_{\Gamma}(x,  u(x),  u(x)) \quad \text{for all }x\in\Gamma.
    \end{equation}

    Since $x\mapsto j^\circ(x,u(x),u(x);-1)$ as well as $x\mapsto j_{\Gamma}^\circ(x,  u(x),  u(x);-1)$ are measurable functions, it follows that $x\mapsto \underline{\eta}(x)$ and $x\mapsto \underline{\zeta}(x)$ are measurable in $\Omega$ and $\Gamma$, respectively, and in view of the growth conditions (J2) on the  Clarke's gradients, we infer $\underline{\eta}\in  \Lp{r_1'(\cdot)}$,  and $\underline{\zeta}\in L^{r_2'(\cdot)}(\Gamma)$. Taking \eqref{609}--\eqref{612} into account, from \eqref{608} we get
    \begin{equation*}
        \langle Au, v-u\rangle+\int_{\Omega} \underline{\eta}(v-u)\,\diff x +\int_{\Gamma}\underline{\zeta}(v- u) \,\diff\sigma\ge 0 \quad\text{for all } v\in u\wedge K,
    \end{equation*}
    which together with (\ref{610}) and (\ref{612}) proves that $u$ is a subsolution of (\ref{605}). By  similar arguments,  one shows that $u$ is also a supersolution of (\ref{605}), which completes the proof.
\end{proof}

\section{Discontinuous Multi-Valued Problems}\label{section_7}

In this section we study discontinuous multi-valued problems. For this purpose, let $j\colon\Omega\times \R\times\R\to\R$ and $j_\Gamma\colon\Gamma\times \R\times\R\to\R$ be given functions such that both are locally Lipschitz continuous with respect to the third argument. We denote by $s\mapsto \partial j(x,r,s)$ and $s\mapsto \partial j_\Gamma(x,r,s)$ Clarke's generalized gradient of $j$ and $j_\Gamma$ with respect to their third variable. Note that we do not suppose any continuity assumptions on $r\mapsto j(x,r,s)
$ and $r\mapsto j_\Gamma (x,r,s)$. This leads to multi-valued functions $f\colon\Omega \times \R\to 2^\R$ and $f_\Gamma\colon\Gamma\times \R\to 2^{\R}$ given by
\begin{align}
    \label{multi-valued-1}
    f(x,s)=\partial (x,s,s)
    \quad\text{and}\quad
    f_\Gamma(x,s)=\partial j_\Gamma(x,s,s).
\end{align}

Based on Proposition \ref{P1} we know that $f\colon \Omega \times \R\to \K(\R)\subset 2^{\R}\setminus\{\emptyset\}$ and $f_\Gamma \colon \Gamma\times\R\to \K(\R)\subset 2^{\R}\setminus\{\emptyset\}$.

The precise problem is stated as follows: Find $u\in K\subset \WH$ and $\tau_1\in C(\close)$, $\tau_2\in C(\Gamma)$, $1<\tau_1(x)<p^*(x)$ for all $x\in\close$, $1<\tau_2(x)<p_*(x)$ for all $x\in\Gamma$ with $\eta\in\Lp{\tau'_1(\cdot)}$, $\zeta\in L^{\tau'_2(\cdot)}(\Gamma)$ such that
\begin{align}
    \label{problem_discontinuous}
    \begin{cases}
        \eta \in \F(u), \ \zeta \in \F_\Gamma(u), \\[1ex]
        \ds \l\langle A(u),v-u\r\rangle
        +\into \eta(v-u) \,\diff x+\int_\Gamma \zeta(v-u)\,\diff \sigma \geq 0\quad\text{for all }v\in K,
    \end{cases}
\end{align}
where $\F$ and $\F_\Gamma$ are the multi-valued Nemytskij operators generated by the multi-valued functions given in \eqref{multi-valued-1}, that is,
\begin{align*}
    \F(u) &=\l\{\eta\colon\Omega\to \R\,:\,\eta \text{ is measurable in }\Omega \text{ and }\right.\\
    &\qquad\qquad \qquad\qquad \left.\eta(x) \in \partial j(x,u(x),u(x)) \text{ for a.\,a.\,}x\in\Omega\r\},             \\
    \F_\Gamma(u) & =\l\{\zeta\colon\Gamma\to \R\,:\,\zeta \text{ is measurable on }\Gamma \text{ and }\right.\\
    &\qquad\qquad \qquad\qquad \left.\zeta(x) \in \partial j_\Gamma (x,u(x),u(x)) \text{ for a.\,a.\,}x\in\Gamma\r\}.
\end{align*}

\begin{remark}
    Note that \eqref{problem_discontinuous} can be equivalently written in the form: $\eta \in \F(u)$, $\zeta \in \F_\Gamma(u)$ and
    \begin{align*}
        \l\langle A(u)+i^*_{\tau_1(\cdot)}\eta+i^*_{\tau_2(\cdot)}\zeta,v-u \r\rangle \geq 0 \quad\text{for all }v\in K.
    \end{align*}
\end{remark}

\begin{definition}
    Let $\Omega\subset\R^N$, $N\geq 1$, be a nonempty measurable set. A function $f\colon \Omega\times \R^m\to \R$, $m\geq 1$, is called superpositionally measurable (or sup-measurable) if the function $x\mapsto f(x,u_1(x),\cdots, u_m(x))$ is measurable in $\Omega$ whenever the component functions $u_i\colon\Omega \to \R$ of $u=(u_1,\ldots,u_m)$ are measurable.
\end{definition}

We suppose the following assumptions on the data.

\begin{enumerate}
    \item[\textnormal{(H1)}]
        Let $j$ and $j_\Gamma$ be superpositionally measurable, that is, if $x\mapsto v(x)$ and $x\mapsto u(x)$ are measurable in $\Omega$, then $x\mapsto j(x,v(x),u(x))$ is measurable in $\Omega$ and if $x\mapsto v(x)$ and $x\mapsto u(x)$ are measurable in $\Gamma$, then $x\mapsto j_\Gamma(x,v(x),u(x))$ is measurable in $\Gamma$.
    \item[\textnormal{(H2)}]
        Let $\underline{u}$ and $\overline{u}$ be sub- and supersolutions of  \eqref{problem_discontinuous} such that $\underline{u}\leq \overline{u}$. There exist $k_\Omega \in L^{\tau_1'(\cdot)}(\Omega)$ and $k_\Gamma\in L^{\tau_2'(\cdot)}(\Gamma)$ with $\tau_1\in C(\close)$, $\tau_2\in C(\Gamma)$, $1<\tau_1(x)<p^*(x)$ for all $x\in\close$, $1<\tau_2(x)<p_*(x)$ for all $x\in\Gamma$ such that
        \begin{align*}
            |\eta|  & \leq k_\Omega \quad\text{for all }\eta\in\partial j(x,r,s), \text{for all }r,s\in [\underline{u}(x),\overline{u}(x)], \text{ for a.\,a.\,}x\in\Omega,           \\
            |\zeta| & \leq k_\Gamma \quad\text{for all }\zeta\in\partial j_\Gamma(x,r,s), \text{for all } r,s\in [\underline{u}(x),\overline{u}(x)], \text{ for a.\,a.\,}x\in\Gamma.
        \end{align*}
    \item[\textnormal{(H3)}]
        The functions $s\mapsto j(x,r,s)$ and $s\mapsto j_\Gamma(x,r,s)$ are locally Lipschitz continuous for all $r\in \R$, for a.\,a.\,$x\in\Omega$ and for a.\,a.\,$x\in\Gamma$, respectively. The functions $r\mapsto j^\circ(r,s;1)$ and $r\mapsto j_\Gamma^\circ(r,s,1)$ are decreasing for all $s\in\R$ for a.\,a.\,$x\in\Omega$ and for a.\,a.\,$x\in\Gamma$, respectively and the functions $r\mapsto j^\circ(r,s;-1)$ and $r\mapsto j_\Gamma^\circ(r,s,-1)$ are increasing for all $s\in\R$ for a.\,a.\,$x\in\Omega$ and for a.\,a.\,$x\in\Gamma$, respectively.
\end{enumerate}

We have the following existence and enclosure result of extremal solutions for  \eqref{problem_discontinuous}.

\begin{theorem}
    \label{theorem-discontinuous}
    Let hypotheses \textnormal{(H0)}, \textnormal{(H1)--(H3)} and the lattice condition
    \begin{align*}
        K\wedge K\subset K\quad \text{and} \quad K\vee K\subset K
    \end{align*}
    be satisfied. Then there exist the greatest solution $u^*$ and the smallest solution $u_*$ of problem \eqref{problem_discontinuous} within the ordered interval $[\underline{u},\overline{u}]$.
\end{theorem}

\begin{proof}
    First, we point out that the multi-valued functions $f$ and $f_\Gamma$ defined in \eqref{multi-valued-1} are no longer upper semicontinuous and so the corresponding Nemytskij operators $\F$ and $\F_{\Gamma}$ are not pseudomonotone in general. Hence, we cannot apply Theorem \ref{surjectivitytheorem}. Instead we will make use of a fixed point argument based on Theorem \ref{theorem-fixed-point} combined with the existence and comparison results provided by Theorems \ref{T102} and \ref{T103}.

    {\bf Step 1:} An Auxiliary problem

    First we choose a fixed $v \in [\underline{u},\overline{u}]$ being a supersolution and a fixed $w \in [\underline{u},\overline{u}]$ being a subsolution of \eqref{problem_discontinuous}. We define
    \begin{align*}
        j^v(x,s) & =j(x,v(x),s), \quad
        j_\Gamma^v(x,s)=j_\Gamma(x,v(x),s),\\
        f^v(x,s)&=\partial j(x,v(x),s),\quad
        f_\Gamma^v(x,s)=\partial j_\Gamma(x,v(x),s)
    \end{align*}
    and
    \begin{align*}
        j_w(x,s)&=j(x,w(x),s),\quad
        j_{w,\Gamma}(x,s)=j_\Gamma(x,w(x),s),\\
        f_w(x,s)&=\partial j(x,w(x),s), \quad
        f_{w,\Gamma}(x,s)=\partial j_\Gamma (x,w(x),s).
    \end{align*}
    Here $s\mapsto f^v(x,s)$, $s\mapsto f^v_\Gamma(x,s)$, $s\mapsto f_w(x,s)$ and $s\mapsto f_{w,\Gamma}(x,s)$ stand for Clarke's generalized gradients and from Proposition \ref{P1} (iii) we know that these are upper semicontinuous. Furthermore, we denote by $\F^v, \F^v_\Gamma, \F_w$ and $\F_{w,\Gamma}$ the multi-valued Nemytskij operators related to $f^v, f^v_\Gamma, f_w$ and $f_{w,\Gamma}$, respectively. We introduce the following auxiliary problems:
    \begin{align}
        \label{aux-1}
        \begin{cases}
            \eta\in \F^v(u), \ \zeta\in \F^v_\Gamma(u), \\[1ex]
            \ds \l\langle A(u), \hat{v}-u \r\rangle  +\into \eta(\hat{v}-u) \,\diff x+\int_\Gamma \zeta(\hat{v}-u)\,\diff \sigma \geq 0\quad\forall \hat{v}\in K,
        \end{cases}
    \end{align}
    \begin{align}
        \label{aux-2}
        \begin{cases}
            \eta\in \F_w(u), \ \zeta\in \F_{w,\Gamma}(u), \\[1ex]
            \ds \l\langle A(u), \hat{v}-u \r\rangle  +\into \eta(\hat{v}-u) \,\diff x+\int_\Gamma \zeta(\hat{v}-u)\,\diff \sigma \geq 0\quad\forall \hat{v}\in K.
        \end{cases}
    \end{align}
    Applying hypothesis (H3) we easily see that $\underline{u}, v$ are sub- and supersolutions of \eqref{aux-1} and $w,\overline{u}$ are sub- and supersolutions of \eqref{aux-2}. Moreover, due to (H1) and (H2), the assumptions of Theorems \ref{T102} and \ref{T103} are satisfied. Therefore, there exist the greatest solution $v^*$ and the smallest solution $v_*$ of \eqref{aux-1} within $[\underline{u},v]$ and the greatest solution $w^*$ and the smallest solution $w_*$ of \eqref{aux-2} within $[w,\overline{u}]$. Furthermore, again by using (H3), we can show that $v^*\in [\underline{u},v]$ is a supersolution of \eqref{problem_discontinuous} and $w_*\in  [w,\overline{u}]$ is a subsolution of \eqref{problem_discontinuous}. This can be shown as it was done in \cite[Lemma 4.1]{Carl-2013}.

    {\bf Step 2:} Definition of fixed-point operators

    We define the following sets:
    \begin{align*}
        \V:=\l\{v\in \WH :v \in [\underline{u},\overline{u}] \text{ and } v \text{ is a supersolution of problem }\eqref{problem_discontinuous}\r\}, \\
        \W:=\l\{w\in \WH :w \in [\underline{u},\overline{u}] \text{ and } w \text{ is a subsolution of problem }\eqref{problem_discontinuous}\r\}.
    \end{align*}
    Recall that $v^* \in [\underline{u},v]$ is the greatest solution of \eqref{aux-1} and $w_*\in [w,\overline{u}]$ is the smallest solution of \eqref{aux-2}, we know that the operators
    \begin{align*}
        G & \colon \V\to\V,\quad \V\ni v\mapsto v^*=Gv,  \\
        T & \colon \W\to\W,\quad \W\ni w\mapsto w_*=Tv,
    \end{align*}
    are well-defined due to Step 1. As done in \cite[Lemma 4.2]{Carl-2013}, using again (H3), one can show that $G\colon\V\to\V$ is an increasing operator, that is, $v_1\leq v_2$ implies $Gv_1 \leq Gv_2$. Similarly, $T\colon\W\to\W$ turns out to be increasing as well.

        {\bf Step 3:} Fixed-point argument

    Using again hypothesis (H3) we are able to show that the range $G(\V)$ of $G$ has an upper bound in $\V$ and decreasing sequences of $G(\V)$ converge weakly in $\V$. The proof is similar to the one in \cite[Lemma 4.5]{Carl-2013} by using Propositions \ref{proposition_embeddings} and \ref{prop1}. Similarly, we show that the range $T(\W)$ of $T$ has a lower bound in $\W$ and increasing sequences of $T(\W)$ converge weakly in $\W$. Now we can apply Theorem \ref{theorem-fixed-point} to $T\colon \W\to\W$ to get a smallest fixed point and to $G\colon \V\to\V$ to get a greatest fixed point. By Definition of $G$, $u \in [\underline{u},\overline{u}]$ is a fixed point of $G$ if and only if $u$ is a solution of \eqref{problem_discontinuous}. Similar can be said about $T\colon\W\to\W$.
    This finishes the proof.
\end{proof}

\section{Construction of Sub-supersolution for a Multi-Valued Obstacle  Problem}\label{section_8}
As an application of the results of the preceding sections, in this section we consider the following obstacle problem when $\mathcal{F}_\Gamma=0$:
\begin{align}
    \label{O1}
    u\in K\,:\, 0\in Au+\partial I_K(u)+\mathcal{F}(u)\quad\text{in }W^{1,\mathcal{H}}_0(\Omega)^*,
\end{align}
where $A$ is the double-phase operator given by \eqref{103} and $K$ is defined by
\begin{equation}
    \label{O2}
    K=\l\{u\in W^{1,\mathcal{H}}_0(\Omega)\,:\, u(x)\geq \psi(x)\text{ a.\,e.\,in }\Omega\r\}.
\end{equation}

Under hypotheses (F1) and (F2) the multi-valued function $f\colon\Omega\times \R\to 2^{\R}\setminus\{\emptyset\}$, which generates the operator $\mathcal{F}$, may be represented by means of single-valued functions $f_i\colon \Omega\times \R\to \R$, $i=1,2$, through
\begin{equation}
    \label{O3}
    f(x,s)=\l[f_1(x,s), f_2(x,s)\r] \quad\text{for all }(x,s)\in \Omega\times \R,
\end{equation}
where $s\mapsto f_1(x,s)$ is a (single-valued) lower semicontinuous function, $s\mapsto f_2(x,s)$ is an (single-valued) upper semicontinuous function, and\\ $x\mapsto f_i(x,u(x))$ is a measurable function for any measurable function $x\mapsto u(x)$. We assume the following hypotheses on the function $\psi$ representing the obstacle, and on $f_i$:
\begin{enumerate}
    \item[\textnormal{(H$\psi$)}]
        $\psi\in \WH$ with trace $\psi|_{\partial\Omega}\le 0$ on $\partial\Omega$ and there exists $c_{\psi}>0$ such that
        \begin{align*}
            \psi(x)\leq c_{\psi}\quad\text{for a.\,e.\,}x\in\Omega.
        \end{align*}
    \item[\textnormal{(Hf)}]
        There exist $k_i\in L^{r_1'(\cdot)}(\Omega)$, $i=1,2,$ such that
        \begin{align*}
            f_1(x,s)\leq k_1(x)
            \quad \text{and}\quad
            f_2(x,s)\ge k_2(x) \quad \text{for all }(x,s)\in \Omega\times \R.
        \end{align*}
\end{enumerate}

\begin{remark}
    Hypothesis \textnormal{(H$\psi$)} implies that $K\neq \emptyset$. We note that hypothesis \textnormal{(Hf)} not necessarily implies boundedness of the multi-valued function $f$ given by \eqref{O3}, since the $f_i$ are only one-sided bounded with respect to $s$.
\end{remark}

Let $u_i \in W^{1,\mathcal{H}}_0(\Omega)\cap \Linf$, $i=1,2$, be the unique (weak) solution of the Dirichlet problem
\begin{equation}
    \label{O4}
    Au_i=-k_i\quad\mbox{in }\Omega,\quad u=0\quad\mbox{on }\partial\Omega.
\end{equation}
Note that the boundedness of $u_i$ can be shown similar to \cite{Winkert-Zacher-2011}, due to the embedding $W^{1,\mathcal{H}}_0(\Omega)\hookrightarrow \Wpzero{p(\cdot)}$, see Proposition \ref{proposition_embeddings}(i).

Our existence and comparison result for the obstacle problem is as follows.
\begin{theorem}
    \label{T-O}
    Assume hypotheses \textnormal{(H0)}, \textnormal{(F1)}, \textnormal{(F2)}, \textnormal{(H$\psi$)}, and \textnormal{(Hf)}. Then the obstacle problem \eqref{O1}, \eqref{O2} has a solution $u$ satisfying $u_1(x)\le u(x)\le u_2(x)+M$ in $\Omega$ for $M\ge 0$ sufficiently large.
\end{theorem}
\begin{proof}
    We are going to make use of Theorem \ref{T102}. To this end we are going to show that $\underline{u}:=u_1$ and $\overline{u}:= u_2+M$ with $M\ge 0$ sufficiently large, are sub- and supersolutions of \eqref{O1}, \eqref{O2}, respectively.

    Let us first show that $\underline{u}:=u_1$ is in fact a subsolution of \eqref{O1}, \eqref{O2} according to Definition \ref{D102}. Clearly, condition $u_1\vee K\subset K$ is fulfilled. Let $\underline{\eta}(x)=f_1(x,u(x))$, then $\underline{\eta}\in L^{r_1'(\cdot)}(\Omega)$ (note: $1<r_1(x)< p^*(x)$) and $\underline{\eta}(x)\in f(x,u_1(x))$, which is (ii). It remains to verify condition (iii) of Definition \ref{D102} (note: $f_\Gamma=0$), that is,
    \begin{equation}
        \label{O5}
        \langle Au_1,v-u_1\rangle +\int_{\Omega}\eta (v-u_1)\,\diff x\ge 0\quad\text{for all }v\in u_1\wedge K,
    \end{equation}
    where
    \begin{align*}
        \langle Au_1,v-u_1\rangle=\into \l(|\nabla u_1|^{p(x)-2} \nabla u_1+ \mu(x) |\nabla u_1|^{q(x)-2} \nabla u_1\r) \cdot \nabla (v-u_1)\,\diff x.
    \end{align*}
    Since  $v\in u_1\wedge K$ can be represented by $v=u_1\wedge\varphi=u_1-(u_1-\varphi)^+$ for any $\varphi\in K$, inequality \eqref{O5} is equivalent to
    \begin{equation}
        \label{O6}
        \l\langle Au_1,(u_1-\varphi)^+\r\rangle +\int_{\Omega}\eta (u_1-\varphi)^+\,\diff x\le 0\quad\text{for all }\varphi\in K.
    \end{equation}
    Since $(u_1-\varphi)^+\in  \{v\in W^{1,\mathcal{H}}_0(\Omega)\,:\, v\geq 0\}$, inequality \eqref{O6} is fulfilled if $u_1$ satisfies (note: $\eta=f_1(\cdot, u_1)$)
    \begin{equation*}
        \langle Au_1,v\rangle +\int_{\Omega}f_1(x, u_1)v\,\diff x\le 0\quad\text{for all }v\in W^{1,\mathcal{H}}_0(\Omega) \text{ with } v\geq 0.
    \end{equation*}
    In view of (Hf) and \eqref{O4} we get $-k_1(x)+f_1(x,u_1)\le 0$, and
    \begin{align*}
        \langle Au_1,v\rangle +\int_{\Omega}f_1(x, u_1)v\,\diff x=\int_{\Omega}\l(-k_1(x)+f_1(x,u_1)\r)v\,\diff x\leq 0
    \end{align*}
    for all $v\in W^{1,\mathcal{H}}_0(\Omega)$ with $v\geq 0$, which proves (iii) of Definition \ref{D102}, and thus $\underline{u}=u_1$ is a subsolution.

    Now let us show that $\overline{u}=u_2+M$ is a supersolution according to Definition \ref{D103} for $M\ge 0$ sufficiently large. From \eqref{O4} we see that $\overline{u}=u_2+M\in \WH$ is the unique solution of the Dirichlet problem
    \begin{equation}
        \label{O8}
        A\overline{u}=-k_2\quad\mbox{in }\Omega,\quad \overline{u}=M\quad\mbox{on }\partial\Omega,
    \end{equation}
    and thus $\overline{u}\wedge K\subset W^{1,\mathcal{H}}_0(\Omega)$. Moreover, since $u_2\in L^\infty(\Omega)$, we get from (H$\psi$) that $\overline{u}=u_2+M\ge c_\psi\ge \psi$, which yields $\overline{u}\wedge K\subset K$ satisfying (i) of Definition \ref{D103}. Set $\overline{\eta}=f_2(\cdot,\overline{u})$, then  $\overline{\eta}\in L^{r_1'(\cdot)}(\Omega)$ and $\overline{u}(x)\in f(x,\overline{u}(x))$, which is (ii). It remains to show (iii) of Definition \ref{D103}, that is,
    \begin{equation}
        \label{O9}
        \langle A\overline{u},v-\overline{u}\rangle +\int_{\Omega}\overline{\eta} (v-\overline{u})\,\diff x\ge 0\quad\text{for all }v\in \overline{u}\vee K.
    \end{equation}
    For $v\in \overline{u}\vee K$ we have $v=\overline{u}\vee \varphi=\overline{u}+(\varphi-\overline{u})^+$, $\varphi\in K$, and thus \eqref{O9} is equivalent to
    \begin{equation}
        \label{O10}
        \l\langle A\overline{u},(\varphi-\overline{u})^+\r\rangle +\int_{\Omega}\overline{\eta} (\varphi-\overline{u})^+\,\diff x\ge 0\quad\text{for all }\varphi\in K.
    \end{equation}
    As $\overline{u}|_{\partial\Omega}=M\ge 0$, it follows that $(\varphi-\overline{u})^+\in \{v\in W^{1,\mathcal{H}}_0(\Omega)\,:\, v\ge 0\}$, hence inequality \eqref{O10} holds true if the following inequality can be verified:
    \begin{equation*}
        \langle A\overline{u},v\rangle +\int_{\Omega}\overline{\eta} v\,\diff x\ge 0\quad\text{for all }v\in W^{1,\mathcal{H}}_0(\Omega) \text{ with }  v\ge 0.
    \end{equation*}
    Taking (Hf) and \eqref{O8} into account, we get $\overline{\eta}-k_2=f_2(\cdot,\overline{u}) -k_2 \ge 0$ and thus
    \begin{equation*}
        \langle A\overline{u},v\rangle +\int_{\Omega}\overline{\eta} v\,\diff x= \int_{\Omega}\l(-k_2+f_2(\cdot,\overline{u})\r) v\,\diff x\geq 0
    \end{equation*}
    for all $v\in W^{1,\mathcal{H}}_0(\Omega)$ with $v\ge 0$, which proves (iii), and thus $\overline{u}=u_2+M$ is a supersolution. Since $u_i\in L^\infty(\Omega)$, $i=1,2$, we obtain by choosing $M\geq 0$ even larger if needed that $\underline{u}=u_1\le u_2+M=\overline{u}$. Applying Theorem \ref{T102} completes the proof.
\end{proof}

One readily verifies that $K$ given by \eqref{O2} satisfies the lattice condition
\begin{align*}
    K\wedge K\subset K\quad \text{and} \quad K\vee K\subset K.
\end{align*}
Hence, as a conclusion of Theorem \ref{T103} we obtain the following characterization of the set $\mathcal{S}$ of all solutions of \eqref{O1},\eqref{O2} lying within the order interval $[\underline{u},\overline{u}]$.

\begin{corollary}
    \label{C-O1}
    Under the hypotheses of Theorem \ref{T-O}, the solution set  $\mathcal{S}$  of \eqref{O1}, \eqref{O2} is a compact subset of  $W^{1,\mathcal{H}}_0(\Omega)$ and possesses a smallest and a greatest solution.
\end{corollary}


\end{document}